\documentclass[10pt]{amsart}
\usepackage{amsmath,amssymb}
\calclayout
\usepackage{enumerate}
\usepackage{color}
\usepackage{graphicx}
\usepackage{dsfont}

\newtheorem{theorem}{Theorem}[section]

\newtheorem{lemma}[theorem]{Lemma}
\newtheorem{corollary}[theorem]{Corollary}
\newtheorem{assumption}[theorem]{Assumption}

\theoremstyle{definition}

\newtheorem{remark}[theorem]{Remark}
\newtheorem{example}[theorem]{Example}

\newcommand{\R}{\mathbb{R}}

\newcommand{\N}{\mathbb{N}}

\renewcommand{\epsilon}{\varepsilon}

\usepackage{accents}
\newlength{\dhatheight}
\newcommand{\doublehat}[1]{%
    \settoheight{\dhatheight}{\ensuremath{\hat{#1}}}%
    \addtolength{\dhatheight}{-0.35ex}%
    \hat{\vphantom{\rule{1pt}{\dhatheight}}%
    \smash{\hat{#1}}}}

\newcommand{\as}{\mbox{-a.s.}}
\DeclareMathOperator{\sgn}{sign}

\DeclareMathOperator{\supp}{supp}

\newcommand{\1}{\mathbf{1}}

\newcommand{\cB}{\mathcal{B}}

\newcommand{\cD}{\mathcal{D}}

\newcommand{\cM}{\mathcal{M}}

\newcommand{\cP}{\mathcal{P}}

\newcommand{\cX}{\mathcal{X}}

\usepackage[pdfborder={0 0 0}]{hyperref}
\hypersetup{
  urlcolor = black,
  pdfauthor = {Marcel Nutz, Johannes Wiesel},
  pdfkeywords = {martingale Schrodinger bridge, potentials, duality},
  pdftitle = {On a Martingale Schr\"odinger Bridge between Two Distributions},
  pdfsubject = {On a Martingale Schr\"odinger Bridge between Two Distributions},
  pdfpagemode = UseNone
}


\begin{document}
\title{On the Martingale Schr\"odinger Bridge between Two Distributions}

\date{\today}

\author{Marcel Nutz}
\thanks{MN acknowledges support by NSF Grants DMS-1812661,  DMS-2106056.}
\address[MN]{Departments of Statistics and Mathematics, Columbia University,  1255 Amsterdam Avenue, New York, NY 10027, USA}
\email{mnutz@columbia.edu}

\author{Johannes Wiesel}
\thanks{JW acknowledges support by NSF Grant DMS-2345556.}
\address[JW]{Department of Mathematics, University of Copenhagen, Universitetsparken 5, 2100 Copenhagen, Denmark}
\email{wiesel@math.ku.dk}

\maketitle

\vspace{-2em}

\begin{abstract}
We study a martingale Schr\"odinger bridge problem: given two probability distributions on the real line, find their martingale coupling with minimal relative entropy. Our main result provides Schr\"odinger potentials for this coupling. Namely, under certain conditions, the log-density of the optimal coupling is given by a triplet of real functions representing the marginal and martingale constraints. The potentials are also described as the solution of a dual problem.
\end{abstract}

\section{Introduction}

The Schr\"odinger bridge problem originates in a question of Schr\"odinger~\cite{Schrodinger.31} about the most likely evolution of a particle system from a distribution~$\mu$ into another distribution~$\nu$. This dynamic problem can be reduced to the \emph{static} Schr\"odinger bridge problem
\begin{align}\label{eq:staticSB}
  \inf_{\pi\in\Pi(\mu,\nu)} H(\pi|R)
\end{align} 
where $H(\cdot|R)$ denotes relative entropy with respect to the reference measure~$R$ and $\Pi(\mu,\nu)$ denotes the set of couplings; i.e., joint distributions with marginals $\mu$ and $\nu$. See \cite{Follmer.88,Leonard.14} for surveys. The problem~\eqref{eq:staticSB} includes the entropically regularized optimal transport (EOT) problem
\begin{align}\label{eq:EOTintro}
  \inf_{\pi\in\Pi(\mu,\nu)} \int c(x,y) \, \pi(dx,dy) +  H(\pi|\mu\otimes\nu)
\end{align} 
where $c$ is a given cost function, by choosing $dR\propto e^{-c}d(\mu\otimes\nu)$. %
EOT is a regularized version of Monge--Kantorovich optimal transport. Thanks to its computational and statistical properties, EOT has been the subject of hundreds of works in the last few years (see, e.g., the monograph~\cite{PeyreCuturi.19} for extensive references).  A central result in EOT states that a coupling~$\pi$ is optimal for~\eqref{eq:EOTintro} if and only if its density is of the form
\begin{align}\label{eq:EOTpotentials}
\frac{d\pi}{d(\mu\otimes\nu)}(x,y)=e^{f(x)+g(y)-c(x,y)}
\end{align} 
for some functions $f,g$ which are called the \emph{Schr\"odinger potentials} and can also be seen as the solution of the dual problem of~\eqref{eq:EOTintro}. Specifically, $(f,g)$ exist in $L^{1}(\mu)\times L^{1}(\nu)$ if $c\in L^{1}(\mu\otimes\nu)$; see, e.g., \cite{Nutz.20}. Potentials are key objects for most aspects of EOT; we give just three examples: On the computational side, potentials are the output of the Sinkhorn--Knopp algorithm, the Gauss--Seidel iteration commonly used to compute EOT~\cite{PeyreCuturi.19}. On the statistical side, sample complexity bounds of EOT are based on growth properties of the potentials \cite{genevay2019sample}. In a mathematical application, regularity properties of the potentials are used to prove a generalization of Caffarelli's contraction theorem~\cite{ChewiPooladian.23}. The rigorous construction of the potentials is surprisingly delicate and was achieved for general costs by~\cite{FollmerGantert.97, RuschendorfThomsen.93, RuschendorfThomsen.97}. Counterexamples in those references highlight some pitfalls in earlier results, related to integrability and convergence issues. 

In this paper we study a \emph{martingale Schr\"odinger bridge} problem, first introduced by~\cite{HenryLabordere.19} in financial mathematics. Financially, the martingale constraint for couplings corresponds to a risk-neutral pricing model that is calibrated to given marginals. The marginals, in turn, represent observed prices of call options. Starting with \cite{BeiglbockHenryLaborderePenkner.11,GalichonHenryLabordereTouzi.11,Hobson.98}, the  optimal transport problem with martingale constraint (\emph{martingale optimal transport}) has been studied in great detail over the last decade; see, e.g., \cite{BeiglbockJourdainMargheritiPammer.22, Wiesel.23} for extensive references. 
The martingale Schr\"odinger bridge problem can be seen as entropically regularized martingale optimal transport. %
Specifically, let $\mu,\nu$ be distributions on $\R$ with finite first moment. A coupling $\pi\in\Pi(\mu,\nu)$ is a martingale if its disintegration $\pi(dx,dy)=\mu(dx)\otimes\pi_{x}(dy)$ satisfies $\int (y-x)\,\pi_x(dy)=0$ for $\mu$-a.e.\ $x\in\R$. Let $\mathcal{M}(\mu,\nu)$ be the set of martingale couplings of~$\mu$ and~$\nu$, then we study the minimization
\begin{align}\label{eq:MSBintro}
\inf_{\pi\in \mathcal{M}(\mu,\nu)} H(\pi|\mu\otimes\nu).
\end{align}
This is connected to the EOT problem~\eqref{eq:EOTintro}: one may suspect that~\eqref{eq:MSBintro} is equivalent to an EOT problem with a suitably chosen cost function~$c$; namely, a Lagrange multiplier for the martingale constraint.

Assuming that $\mathcal{M}(\mu,\nu)\neq\emptyset$, standard arguments show that~\eqref{eq:MSBintro} admits a unique optimizer~$\pi^{\ast}$. Our main result (Theorem~\ref{thm:main}) is that, under integrability and irreducibility conditions detailed below, this optimizer has a density of the form
\begin{align}\label{eq:densityFormIntro}
\frac{d\pi^*}{d(\mu\otimes\nu)}(x,y)=e^{f(x)+g(y)-h(x)(y-x)}
\end{align}
for some functions $f\in L^1(\mu)$, $g\in L^1(\nu)$, and $h\in L^{0}(\mu)$ with $h(x)(y-x)\in L^{1}(\pi)$ for all $\pi\in\mathcal{M}_{\text{fin}}(\mu,\nu):=\{\pi\in \mathcal{M}(\mu,\nu):\, H(\pi|\mu\otimes \nu)<\infty\}$.
These functions are unique (up to shift by an affine function) and constitute the potentials of our martingale Schr\"odinger bridge problem. Like in~\eqref{eq:EOTpotentials}, the functions $(f,g)$ can be seen as Lagrange multipliers for the marginal constraints $(\mu,\nu)$. The additional term $h(x)(y-x)$ in~\eqref{eq:densityFormIntro} is a Lagrange multiplier for the martingale constraint. We see that~\eqref{eq:MSBintro} is formally equivalent to the EOT problem~\eqref{eq:EOTintro} with cost $c(x,y)=h(x)(y-x)$. To be precise, the equivalence holds if $h(x)(y-x)$ has good integrability. However we will see that this can fail: in Example~\ref{ex:strictSeparationHelps}, the positive part of that function is not integrable under $\mu\otimes\nu$, and then the cost $c(x,y)=h(x)(y-x)$ does not fall within the setting generally covered by EOT theory. That may give a hint that the potentials cannot be constructed by simply invoking a standard result from EOT.

Existence of the potentials in the aforementioned sense yields a strong duality theorem (Corollary~\ref{co:duality}). Indeed, 
\begin{align*}
  &\inf_{\pi\in \mathcal{M}(\mu,\nu)} H(\pi|\mu\otimes\nu) \\
  &= \sup_{\tilde f,\tilde g,\tilde h} \int \tilde f(x)\,\mu(dx)+ \int \tilde g(y)\,\nu(dy) - \log\int e^{\tilde f(x)+\tilde g(y)-\tilde h(x)(y-x)}\,\mu(dx)\nu(dy)
\end{align*} 
where the supremum is taken over all $\tilde f\in L^1(\mu)$, $\tilde g\in L^1(\nu)$, and $\tilde h\in L^{0}(\mu)$ with $\tilde h(x)(y-x)\in  L^1(\pi)$ for all $\pi\in \mathcal{M}_{\text{fin}}(\mu,\nu)$. Moreover, the dual supremum is attained: the maximizers are precisely the potentials $(f,g,h)$, leading to the formula $\inf_{\pi\in \mathcal{M}(\mu,\nu)} H(\pi|\mu\otimes\nu)=\int f(x)\,\mu(dx)+ \int g(y)\,\nu(dy)$.

First statements about potentials for martingale Schr\"odinger bridges can be found in~\cite{HenryLabordere.19} for a dynamic setting, though those formal statements are not meant to be mathematically rigorous. Compared with the setting of classical Schr\"odinger bridges and EOT, the additional dual term $h(x)(y-x)$ due to the martingale constraint creates an interaction between the variables $x,y$ which turns out to be a substantial difficulty. This difficulty was first highlighted by a counterexample of~\cite{AcciaioLarssonSchachermayer.17}. The example exhibits an $L^{p}$-convergent sequence $g_{n}(y)-h_{n}(x)(y-x)$ with nice functions $g_{n},h_{n}$ where the limit is not of the form $g(y)-h(x)(y-x)$ with an integrable~$g$. The follow-up work~\cite{NutzWieselZhao.22a} showed that this stems from a loss of integrability in the passage to the limit. 
The results closest to the present paper are obtained in~\cite{NutzWieselZhao.22b} for a martingale Schr\"odinger 
bridge problem where only the second marginal~$\nu$ is given (indeed, a pier rather than a bridge). As there is no constraint on the first marginal, the dual takes the form $g(y)-h(x)(y-x)$ without the function~$f$ as in~\eqref{eq:densityFormIntro}. The approach in~\cite{NutzWieselZhao.22b} rests on the flexibility of manipulating the support of the first marginal, a nonstarter in the present setting where~$\mu$ is prescribed. A different type of result is obtained in~\cite{NutzWieselZhao.22a} by algebraic considerations. It implies the algebraic form of the log-density of $\pi^{*}$, but it yields functions that a priori are not integrable (in general, possibly even not measurable). As far as we could see, the result of~\cite{NutzWieselZhao.22a} does not help in establishing the present results.

Instead, we use a completely new approach. We first study an auxiliary problem where the delicate martingale constraint is replaced by a weaker constraint that is easier to handle. We show that the coupling $\hat\pi$ solving that relaxed problem has a density of the desired form~\eqref{eq:densityFormIntro}. Second, we show that $\hat\pi$ is actually a martingale, which implies that $\hat\pi=\pi^{*}$ is also the solution of the martingale Schr\"odinger bridge problem. As a result, the density of $\pi^{*}$ has the desired form. A more detailed outline of the approach can be found in Remark~\ref{rk:proofIdea}. 

A similar approach was adopted in the follow-up work \cite{ChenFanConforti.24} to the present paper, which shows convergence of Sinkhorn iterates of the martingale Schr\"odinger bridge problem. In particular, this establishes the existence of Schr\"odinger potentials. The setting of \cite{ChenFanConforti.24} includes the generalization to reference measures that are not necessarily of product type, which is equivalent to generalizing the problem formulation~\eqref{eq:MSBintro} to
\begin{align*}
\inf_{\pi\in \mathcal{M}(\mu,\nu)} \int c(x,y) \, \pi(dx,dy) + H(\pi|\mu\otimes\nu).
\end{align*}
On the other hand, the proof necessitates an assumption that is significantly more stringent than ours, namely that there exists $\bar{\pi}\in \mathcal{M}(\mu,\nu)$ with uniformly bounded log-density relative to $\mu\otimes \nu$.

Continuing with the broadly related literature, similar potentials also play an important role in~\cite{Guyon.20, Guyon.21} tackling the joint S\&P\,500/VIX calibration problem; here there are additional constraints pertaining to observing prices on VIX derivatives. The existence of integrable potentials is posited, but showing existence from first principles remains an open problem. To the best of our knowledge, the present result is the first to provide integrable potentials for a martingale Schr\"odinger bridge with more than one marginal. In a different direction, \cite{DeMarchHenryLabordere.19} develop a version of Sinkhorn's algorithm for martingale optimal transport using the martingale Schr\"odinger bridge as a regularized version of the optimal transport problem. See also~\cite{GuoObloj.19} for a related algorithm using a different relaxation. In another related direction, \cite{DoldiFrittelli.23, DoldiFrittelliRosazzaGianin.23} study an entropic martingale optimal transport problem different from the present one. Instead of minimizing entropy over couplings, these works consider a martingale version of the transport problem proposed in \cite{LieroMielkeSavare.18}. In that problem, the marginal constraints $\mu,\nu$ are relaxed into an entropic penalty, whereas the entropy of the coupling does not enter the picture.

The remainder of the paper is organized as follows. Section~\ref{se:main} details the problem formulation and states the main results. The proofs are reported in Section~\ref{se:proofs}, while Section~\ref{se:closing} concludes with an example illustrating a technical condition used in the main results.

\section{Main Results}\label{se:main}

We write $\cB(\cX)$ and $\mathcal{P}(\mathcal{X})$ for the collections of Borel sets and probability measures on~$\cX$, respectively. Measures are endowed with the weak topology induced by $C_b(\mathcal{X})$.
Let $\mu,\nu\in\mathcal{P}(\R)$ be probability measures on $\R$ with finite first moments and in convex order, where the latter is denoted by $\mu\preceq_c\nu$ and defined by $\int \varphi(x) \,\mu(dx)\leq \int \varphi(x) \,\nu(dx)$ for any convex function $\varphi:\R\to\R$. This implies that $\mu$ and $\nu$ have the same mean. Without loss of generality, we can then assume that the measures are centered,
\begin{align}\label{eq:centered}
  \int x\,\mu(dx)=\int y\,\nu(dy)=0.
\end{align} 
Their set of couplings is defined by 
\begin{align*}
\Pi(\mu,\nu)=\{\pi\in \mathcal{P}(\R\times \R): \pi(A\times \R)=\mu(A),\, \pi(\R\times A)=\nu(A) \text{ for all }A\in\cB(\R)\}.
\end{align*} 
For $\pi\in \Pi(\mu,\nu)$, we denote by $\pi_x$ the conditional probability distribution of $\pi$ given $x$, meaning that the Markov kernel $(\pi_x)_{x\in \R}$ satisfies $\pi(A\times B)=\int_A \pi_x(B)\,\mu(dx)$ for all $A,B\in\cB(\R)$. 
The set of martingale couplings is defined by
\begin{align*}
\mathcal{M}(\mu,\nu)= \left\{\pi\in \Pi(\mu,\nu): \int (y-x)\,\pi_x(dy)=0\quad \mu\text{-a.s.}\right\}.
\end{align*}
By Strassen's theorem \cite{Strassen.65},  $\mu\preceq_c\nu$ is equivalent to $\mathcal{M}(\mu,\nu)\neq \emptyset$. 
If $\mu=\delta_0$, then $\mathcal{M}(\mu,\nu)=\{\mu\otimes\nu\}$ and all our results are trivial. Hence we may assume without loss of generality that $\mu\neq\delta_0$, which in view of $\mu\preceq_c\nu$ also implies that  $\nu\neq\delta_0$.

We consider the martingale Schr\"odinger bridge problem 
\begin{align}\label{eq:opt}
\inf_{\pi\in \mathcal{M}(\mu,\nu)} H(\pi | \mu\otimes\nu),
\end{align}
where $H$ denotes the relative entropy
\begin{align*}
H(\pi |\mu\otimes\nu)=\begin{cases}
\int \log(d\pi/d(\mu\otimes\nu))\,d\pi\quad &\text{if } \pi\ll \mu\otimes\nu,\\
\infty \quad&\text{otherwise}.
\end{cases}
\end{align*}
For~\eqref{eq:opt} to be well-posed, we clearly require some $\pi\in \mathcal{M}(\mu,\nu)$  with $H(\pi | \mu\otimes\nu)<\infty$. A stronger condition is necessary for our main result~\eqref{eq:opt1} below: if~\eqref{eq:opt1} holds, then the coupling $\pi^{*}\in \mathcal{M}(\mu,\nu)$ is equivalent to $\mu\otimes\nu$, denoted $\pi^{*}\sim \mu\otimes \nu$. We state these necessary conditions as a standing assumption.

\begin{assumption}\label{ass:0} 
There exists $\bar{\pi}\in \mathcal{M}(\mu,\nu)$ such that $$\bar{\pi}\sim \mu\otimes \nu  \quad \text{and} \quad H(\bar{\pi}|\mu\otimes\nu)<\infty.$$
\end{assumption}

The existence of a martingale coupling equivalent to $\mu\otimes\nu$ can be characterized in terms of the marginals~$\mu$ and~$\nu$ as follows.

\begin{remark}\label{rk:irred}
Define the potential function $\varphi_{\rho}$ of $\rho\in\cP(\R)$ by $\varphi_{\rho}(x)=\int |x-y|\,\rho(dy)$. Then $\varphi_{\mu}\leq \varphi_{\nu}$ due to $\mu\preceq_c\nu$. Existence of a martingale coupling  $\bar{\pi}\sim\mu\otimes\nu$ is equivalent to a \emph{strict} inequality on an open interval of full $\mu$-measure. This condition is known in martingale optimal transport theory as irreducibility:  $(\mu,\nu)\in\cP(\R)\times\cP(\R)$ is called irreducible if the set $I:=\{\varphi_{\mu}< \varphi_{\nu}\}$ is connected and satisfies $\mu(I)=1$.

To see the equivalence, note that if $(\mu,\nu)$ is irreducible, then \cite[Lemma~3.3 and its proof]{BeiglbockNutzTouzi.15} show that there exists $\bar{\pi}\in \mathcal{M}(\mu,\nu)$  with $\bar\pi\sim \mu\otimes\nu$. Conversely, if $(\mu,\nu)$ is not irreducible, there exists $x\in\R$ such that $\varphi_{\mu}(x)=\varphi_{\nu}(x)$ and either $\mu((-\infty,x])>0$ and $\mu((x,\infty))>0$, or $\mu((-\infty,x))>0$ and $\mu([x,\infty))>0$.
  In the first case, the set $A:=(-\infty,x]\times (x,\infty)$ satisfies $\pi(A)=0$ for all $\pi\in\mathcal{M}(\mu,\nu)$ by \cite[Theorem~3.2]{BeiglbockNutzTouzi.15} but $(\mu\otimes\nu)(A)>0$, so that $\bar{\pi}$ cannot exist. The second case is analogous with $A:=(-\infty,x)\times [x,\infty)$.
\end{remark} 

Remark~\ref{rk:irred} shows that a strict inequality between the potential functions $\varphi_{\mu},\varphi_{\nu}$ is necessary for our main result to hold. A strict inequality means that the difference of the potential functions is uniformly bounded away from zero on compact subintervals of~$I$. Our proof of the main result uses a slightly stronger condition. Writing $[s_-, s_+]$ for the convex hull of the support of $\mu$, we require that either the difference is uniformly bounded away from zero on $[s_-, s_+]$ , or if the potentials touch at an endpoint of~$[s_-, s_+]$, then at least they should meet at a nonzero angle (which means that~$\nu$ has an atom at that point). This condition can be phrased succinctly as follows, where we write $\supp(\mu)$ for the support of~$\mu$.

\begin{assumption}\label{ass:1} 
Let $s_-=\inf(\supp(\mu))$ and $s_+=\sup(\supp(\mu))$ be the left and right endpoints of the support of~$\mu$. We assume that  $\nu((-\infty, s_-])>0$ and $\nu([s_+, \infty))>0$. 
\end{assumption}

If we denote by $t_{\pm}$ the endpoints of the support of $\nu$, then Assumption~\ref{ass:1} means that either $t_{+}>s_{+}$, or $t_{+}=s_{+}$ and $\nu(\{s_{+}\})>0$, and similarly for $s_{-}$. It follows that $\mu$ has bounded support, though this will not be used directly. In fact, Assumption~\ref{ass:1} will be used only a single time, to show uniform integrability of a certain sequence in Step~2 of the proof of Lemma~\ref{lem:hard}. We do not know if Theorem~\ref{thm:main} holds without Assumption~\ref{ass:1}. Example~\ref{ex:strictSeparationHelps} highlights that certain integrability issues arise when the potentials are not properly separated.

We can now state our main result.

\begin{theorem}[Potentials]\label{thm:main}
Under Assumptions \ref{ass:0} and \ref{ass:1}, there is a unique optimizer $\pi^\ast \in \mathcal{M}(\mu,\nu)$ of \eqref{eq:opt}. It is characterized within $\mathcal{M}(\mu,\nu)$ by having a density of the form
\begin{align}\label{eq:opt1}
\frac{d\pi^*}{d(\mu\otimes\nu)}(x,y)=e^{f(x)+g(y)-h(x)(y-x)}
\end{align}
for some functions $f\in L^1(\mu)$, $g\in L^1(\nu)$ and $h\in L^{0}(\mu)$ with $h(x)(y-x)\in  L^1(\pi)$ for all $\pi\in\mathcal{M}_{\text{fin}}(\mu,\nu):=\{\pi\in \mathcal{M}(\mu,\nu):\, H(\pi|\mu\otimes \nu)<\infty\}$ and $\mu(\{x: xh(x)\le 0\})=1$. We call any such triplet \emph{potentials} of~\eqref{eq:opt}.

The potentials are unique up to an affine shift. More precisely, $(f',g',h')$ are potentials if and only if there are constants $\alpha,\beta\in\R$ such that 
\begin{align*}
  f'(x) - f(x) = \alpha x + \beta \quad \mu\as, \qquad g(y)-g'(y)=\alpha y+\beta \quad \nu\as, \qquad  h-h'= \alpha \quad \mu\as
\end{align*}
\end{theorem}

Theorem~\ref{thm:main} implies a strong duality result.

\begin{corollary}[Strong duality]\label{co:duality}
Define the dual domain
\begin{align*}%
\cD = \Big\{(\tilde f,\tilde g,\tilde h)\in &L^{1}(\mu)\times L^{1}(\nu)\times L^{0}(\mu):\,\tilde{h}(x)(y-x)\in L^1(\pi) \quad\mbox{for all}\quad \pi\in \mathcal{M}_{\text{fin}}(\mu,\nu)\Big \}.
\end{align*}
We have 
\begin{align*}%
&\inf_{\pi\in \mathcal{M}(\mu,\nu)} H(\pi | \mu\otimes\nu)\\
&= \sup_{(\tilde f,\tilde g,\tilde h)\in \cD} \int \tilde f(x)\, \mu(dx) + \int \tilde g(y) \, \nu(dy) - \log \int e^{\tilde f(x)+\tilde g(y)-\tilde h(x)(y-x)}\, \mu(dx)\nu(dy)
\end{align*}
and the argmax of the right-hand side is precisely the set of potentials $(f,g,h)$. 
In particular,
\begin{align}\label{eq:rev1}
\inf_{\pi\in \mathcal{M}(\mu,\nu)} H(\pi | \mu\otimes\nu) 
 = \int f(x)\, \mu(dx) + \int g(y) \, \nu(dy).
\end{align}
\end{corollary}

\begin{proof}
  If $\pi=\pi^{*}$ and $( f, g, h)$ are potentials, the two sides of the first display are equal because the last integral has value $\pi^{*}(\R\times\R)=1$. By the weak duality inequality, shown in Lemma~\ref{le:weakDuality} below, it follows that $(f,g,h)$ are dual maximizers and \eqref{eq:rev1} holds.
Conversely, let $(\tilde f,\tilde g,\tilde h)\in \mathcal{D}$ be dual maximizers. Then strict concavity of the dual problem implies that $(\tilde f,\tilde g,\tilde h)$ are potentials. More precisely, \eqref{eq:jensen_rev} with $\pi=\pi^{*}$ and strict concavity of $x\mapsto \log(x)$ show that $\tilde{f}(x)+\tilde{g}(y)-\tilde{h}(x)(y-x)=f(x)+g(y)-h(x)(y-x)$ $(\mu\otimes\nu)$-a.s. 
\end{proof} 

\begin{remark}
  Theorem~\ref{thm:main} is slightly weaker than its analogue in EOT, where it is known that a coupling with density of the form $\exp(f(x) + g(y))$ is necessarily optimal, even without any integrability conditions on $f$ and $g$  (see~\cite{Nutz.20}). The proof of the latter result exploits a particularity of the form $f(x) + g(y)$ and does not apply here. Using the techniques in~\cite{NutzWieselZhao.22a}, one could show, even without Assumption~\ref{ass:1}, that the optimal coupling $\pi^{*}$ has a density of the form $e^{f(x)+g(y)-h(x)(y-x)}$ for some functions $f,g,h$. Moreover, those functions are unique up to an affine shift, by Lemma~\ref{le:uniqueness}. However, the  integrability of the functions is unclear in general. Assumption~\ref{ass:1} allows us to establish integrability from first principles and formulate a straightforward relationship between the primal and dual problems as stated in Theorem~\ref{thm:main}.
\end{remark}

\section{Proofs}\label{se:proofs}  %

The primal existence and uniqueness are straightforward. Indeed, by the lower semicontinuity of $\pi\mapsto H(\pi| \mu\otimes \nu )$ and weak compactness of $\mathcal{M}(\mu,\nu)$ there exists a minimizer $\pi^\ast\in \mathcal{M}(\mu,\nu)$ for the martingale Schr\"odinger bridge problem~\eqref{eq:opt},
\begin{align*}
\inf_{\pi\in \mathcal{M}(\mu,\nu)} H(\pi | \mu\otimes\nu)=H(\pi^\ast| \mu\otimes\nu),
\end{align*}
and the minimizer is unique by the strict convexity of $\pi\mapsto H(\pi| \mu\otimes\nu)$. Alternately, existence and uniqueness also follow from the general result of \cite{Csiszar.75} on entropy minimization.

The essential part of Theorem \ref{thm:main} is to construct the functions $f,g,h$ together with their integrability properties. As the proof is quite long, the following remark sketches its broad idea.

\begin{remark}[Proof idea]\label{rk:proofIdea}
We try to relax the delicate martingale constraint into a weaker constraint that is easier to handle. Specifically, we consider 
\begin{align}\label{eq:opt3}
\inf_{\pi \in \widehat{\mathcal{M}}(\mu,\nu)} H(\pi |\mu\otimes \nu),
\end{align}
where 
\begin{align}\label{eq:hatM}
\widehat{\mathcal{M}}(\mu,\nu) = \left\{\pi\in \Pi(\mu,\nu):\ x \int (y-x)\,\pi_x(dy) \ge 0\quad \mu\text{-a.s.}\right\}.
\end{align}
Clearly $\widehat{\mathcal{M}}(\mu,\nu)\supseteq \mathcal{M}(\mu,\nu)$. We will establish the desired properties for the optimizer~$\hat\pi$ of~\eqref{eq:opt3}, and then show that it is in fact also the optimizer for~\eqref{eq:opt}. The intuition for this is as follows.  
The measure $\pi=\mu\otimes\nu$ satisfies $\pi_x=\nu$ and hence
\begin{align*}%
\int (y-x)\,\pi_x(dy)=\int y\,\nu(dy) -x=-x, \qquad \pi:=\mu\otimes\nu.
\end{align*}
That is, the sign of the (shifted) barycenter of $\pi_x$ is opposite to the sign of~$x$, which is the exact opposite of the condition in~\eqref{eq:hatM}. As the optimizer $\hat\pi$ of \eqref{eq:opt3} strives to be ``close'' to $\mu\otimes \nu$ by its definition, we may speculate that it saturates the constraint in~\eqref{eq:hatM}, which is to say that $\hat\pi\in \cM(\mu,\nu)$.

Our proof below starts from  the  dual of the convex problem~\eqref{eq:opt3}. We  will show in Lemma \ref{lem:duality} that
\begin{align*}%
\inf_{\pi\in \widehat{\mathcal{M}}(\mu,\nu)} H(\pi| \mu\otimes\nu)=\sup_{h\in C_b(\R):\  x h(x)\le 0} \inf_{\pi\in \Pi(\mu,\nu)} \int h(x)(y-x)\,\pi(dx,dy)+H(\pi| \mu\otimes\nu),
\end{align*}
where $C_b(\R)$ denotes the bounded continuous functions on $\R$. 
Taking a maximizing sequence $(h_m)$ for the supremum, we define $(f_m, g_m)$ as the corresponding Schr\"odinger potentials from EOT theory and argue that the sequence of measures $(f_m,g_m,h_m)_*(\mu\otimes\nu)$ is tight (where $\ast$ denotes pushforward). Denoting the weak limit by $(f,g,h)_*(\mu\otimes\nu)$, we prove that $(f,g,h)$ give rise to the form \eqref{eq:opt1} for $\hat{\pi}$, and that the optimizers $\hat{\pi}$ and $\pi^*$ coincide. 
\end{remark} 

Moving on to the actual proof, we recall the centering~\eqref{eq:centered} as well as $\mu\neq\delta_0$ and $\nu\neq\delta_0$. We also define $\sgn(x):=\1_{\R^+}(x) -\1_{\R^-}(x)$, where 
$$\R^+:=[0,\infty), \qquad \R^-:=(-\infty,0).$$ 
We first introduce a probability measure $P\in \Pi(\mu,\nu)$ which will act as a counterpart of $\mu\otimes\nu$ in the sense that its kernel satisfies inequalities opposite to the ones that $\nu$ (the kernel of $\mu\otimes\nu$) satisfies. 

\begin{lemma}\label{lem:P}
There exists $P\in \Pi(\mu,\nu)$ such that $H(P|\mu\otimes\nu)<\infty$ and $\mu$-a.s.,
$$ \int (y-x)\,P_x(dy)
\begin{cases} 
>0 \quad \text{for }x\ge 0,\\
<0 \quad \text{for }x< 0.
\end{cases}$$
\end{lemma}

\begin{proof}
By Assumption \ref{ass:0}, there exists $\bar{\pi}\in \mathcal{M}(\mu,\nu)$ satisfying $\bar{\pi}\sim \mu\otimes \nu$ and $H(\bar{\pi}|\mu\otimes\nu)<\infty$. In order to construct $P$, we start from $\bar{\pi}$ and carefully manipulate its conditional distributions $\bar{\pi}_x$: for any $x_-<0\le x_+$, we exchange parts of $\bar{\pi}_{x_+}$ and $\bar{\pi}_{x_-}$ in such a way that the marginal laws do not change.

Fix an arbitrary coupling $$\gamma\in \Pi\Big(\tfrac{1}{\mu(\R^+)} \mu|_{\R^+},\tfrac{1}{\mu(\R^-)} \mu|_{\R^-} \Big);$$ 
it exists due to $\mu\neq \delta_0$. We claim that
\begin{align*}
\epsilon(x^+, x^-):= \bar\pi_{x^+}((-\infty, x^-)) \wedge \bar\pi_{x^-}(x^+,\infty))>0 \quad\mbox{$\gamma(dx^+,dx^-)$-a.s.}
\end{align*}
Indeed, suppose that $\bar\pi_{x^+}((-\infty, x^-)) = 0$ on a set of positive $\gamma$-measure, then as $\bar\pi_x\sim \nu$ $\mu$-a.s.\ due to $\bar\pi \sim \mu\otimes \nu$, we have
\begin{align*}
  \bar\pi_{x^+}((-\infty, x^-)) = 0 = \nu((-\infty, x^-)) = 0 = \bar\pi_{x^-}((-\infty, x^-)).
\end{align*}
As $\bar\pi_{x^-}$ has barycenter $x^-$, this implies that $\bar\pi_{x^-}=\delta_{x^-}$. This contradicts that $\bar\pi_x\sim \nu$ and~$\nu$ is not a Dirac measure.

Next, we define 
\begin{align*}
P^+_{x^+,x^-} &:=  \mu(\R^+) \Big(\bar\pi_{x^+} +\epsilon(x^+, x^-) \mu(\R^-) \Big[\frac{\bar\pi_{x^-}|_{(x^+,\infty)}}{\bar\pi_{x^-}((x^+,\infty))} - \frac{\bar\pi_{x^+}|_{(-\infty, x^-)}}{\bar\pi_{x^+}((-\infty, x^-))} \Big] \Big), \\
P^-_{x^+,x-} &:= \mu(\R^-) \Big( \bar\pi_{x^-} +  \epsilon(x^+, x^-) \mu(\R^+) \Big[ \frac{\bar\pi_{x^+}|_{(-\infty, x^-)}}{\bar\pi_{x^+}((-\infty, x^-))} - \frac{\bar\pi_{x^-}|_{(x^+,\infty)}}{\bar\pi_{x^-}((x^+,\infty))}\Big] \Big)
\end{align*}
and set 
\begin{align}\label{eq:chain1}
P(dx,dy) := \int_{\{x^-\in \R^-\}}  P^+_{x,x^-}(dy) \gamma(dx,dx^-) +  \int_{\{x^+\in \R^+\}} P^-_{x^+,x}(dy) \gamma(dx^+,dx).   
\end{align}
 To see that $P\in \Pi(\mu,\nu)$, we consider an arbitrary Borel set $A\subseteq \R$ and compute
\begin{align*}
&P(A \times \R) \\
&= \int_{\{(x,x^-,y)\in A\times\R^-\times\R\}}  P^+_{x,x^-}(dy) \gamma(dx,dx^-) +  \int_{\{(x,x^+,y)\in A\times\R^+\times \R \}}  P^-_{x^+,x}(dy) \gamma(dx^+,dx)\\
&=  \mu(\R^+) \int_{\{(x,x^-)\in A\cap\R^+\times \R^-\}} \gamma(dx,dx^-) + \mu(\R^-)  \int_{\{(x,x^+)\in A\cap \R^- \times \R^+\}} \gamma(dx^+,dx) \\
&= \mu(\R^+) \frac{1}{\mu(\R^+)} \mu(A\cap \R^+) +  \mu(\R^-)\frac{1}{\mu(\R^-)} \mu(A\cap \R^-) = \mu(A)
\end{align*}
as well as
\begin{align*}
&P(\R \times A) \\
&= \int_{\{(x,x^-, y)\in \R^+\times\R^-\times A\}}  P^+_{x,x^-}(dy) \gamma(dx,dx^-) +  \int_{\{ (x,x^+, y)\in \R^-\times\R^+\times A\}} P^-_{x^+,x}(dy) \gamma(dx^+,dx)\\
&= \int_{\R^+} \bar\pi_x(A)\,\mu(dx) + \int_{\R^-} \bar\pi_x(A)\,\mu(dx)\\
&\quad +  \int_{\{(x^+,x^-)\in \R^+\times\R^-\}}  \epsilon(x^+, x^-)  \mu(\R^-)  \mu(\R^+)  \Big( \Big[\frac{\bar\pi_{x^-}(A\cap (x^+,\infty))}{\bar\pi_{x^-}((x^+,\infty))} - \frac{\bar\pi_{x^+}(A\cap (-\infty, x^-))}{\bar\pi_{x^+}((-\infty, x^-))} \Big]   \\
&\qquad\qquad\qquad + \Big[ \frac{\bar\pi_{x^+}(A\cap (-\infty, x^-))}{\bar\pi_{x^+}((-\infty, x^-))} - \frac{\bar\pi_{x^-}(A\cap (x^+,\infty))}{\bar\pi_{x^-}((x^+,\infty))}\Big]\Big)\, \gamma(dx^+,dx^-)=\nu(A).
\end{align*}
Furthermore,
$$ \int (y-x)\,P_x(dy)
\begin{cases} 
>0 \quad \text{for }x\ge 0,\\
<0 \quad \text{for }x< 0
\end{cases}$$
follows directly from the  construction of~$P$.

It remains to show $H(P|\mu\otimes\nu)<\infty$. To that end, write 
\begin{align}\label{eq:chain2}
(\mu\otimes \nu) (dx,dy) =  \mu(\R^+) \int_{\{x^-\in \R^-\}}   \nu(dy) \gamma(dx,dx^-) +  \mu(\R^-) \int_{\{x^+\in  \R^+\}}  \nu(dy) \gamma(dx^+,dx)
\end{align}
and recall that for any $R,Q\in\cP(\R^2)$ with first marginals $R^1,Q^1$ we have the ``chain rule''
\begin{align}\label{eq:chain3}
H(R|Q) = H(R^1|Q^1)+ \int H(R_x|Q_x)\,R^1(dx).
\end{align}
Comparing \eqref{eq:chain1} and \eqref{eq:chain2} and using joint convexity of $H(\cdot|\cdot)$ twice, we have
\begin{align*}
H(P | \mu\otimes\nu) &\le \mu(\R^+) \int_{\{x^-\in \R^-\}} H\Big(  \int \frac{P^+_{\cdot ,x^-}}{\mu(\R^+)}  \gamma_{x^-}(\cdot)  \Big| \nu \otimes \gamma_{x^-}\Big) \frac{1}{\mu(\R^-)} \mu(dx^-)\\
&+ 
\mu(\R^-) \int_{\{x^+\in  \R^+\}} H\Big( \int \frac{P^-_{x^+,\cdot }}{\mu(\R^-)}  \gamma_{x^+} (\cdot ) \Big| \nu \otimes \gamma_{x^+}\Big) \frac{1}{\mu(\R^+)} \mu(dx^+).
\end{align*}
Using \eqref{eq:chain3}, we obtain by convexity of relative entropy that
\begin{align*}
H(P | \mu\otimes\nu) &\le  \mu(\R^+) \int_{\{(x,x^-)\in \R^+\times\R^- \}} H( P^+_{x,x^-}/\mu(\R^+) | \nu)\, \gamma(dx,dx^-)\\
 &\qquad+ \mu(\R^-) \int_{\{(x^+,x)\in \R^+\times\R^- \}}    H( P^-_{x^+,x}/\mu(\R^-) |  \nu)\, \gamma(dx^+,dx)\\
& \le \mu(\R^+) \int_{\{(x,x^-)\in \R^+\times\R^- \}} H(   \bar\pi_x | \nu) + H(  \mu(\R^-) \bar\pi_{x^-} |  \nu) \, \gamma(dx,dx^-) \\
&\qquad +  \mu(\R^-) \int_{\{(x^+,x)\in \R^+\times\R^- \}}  H(  \bar\pi_{x} | \nu) + H(   \mu(\R^+) \bar\pi_{x^+} |  \nu)  \, \gamma(dx^+,dx)+C\\
&\le 2 \int H(\bar\pi_x|\nu)\,\mu(dx)+C<\infty,
\end{align*}
where $C>0$ is chosen such that $x\log(x)+C\ge 0$.
This concludes the proof.
\end{proof}

Next, we consider the convex set $\widehat{\mathcal{M}}(\mu,\nu)\supseteq \mathcal{M}(\mu,\nu)$ announced in~\eqref{eq:hatM},
\begin{align*}
\widehat{\mathcal{M}}(\mu,\nu) := \left\{\pi\in \Pi(\mu,\nu):\ x \int (y-x)\,\pi_x(dy) \ge 0\quad \mu\text{-a.s.}\right\}.
\end{align*}
Using the same arguments as for $\mathcal{M}(\mu,\nu)$, the set $\widehat{\mathcal{M}}(\mu,\nu)$ is weakly compact (see e.g. \cite[Proof of Prop.~2.4.]{BeiglbockHenryLaborderePenkner.11}), so there is a unique $\hat{\pi}\in \widehat{\mathcal{M}}(\mu,\nu)$ such that
\begin{align}\label{eq:inclusion}
H(\hat{\pi}| \mu\otimes\nu)=\inf_{\pi\in \widehat{\mathcal{M}}(\mu,\nu)} H(\pi | \mu\otimes\nu)\le \inf_{\pi\in \mathcal{M}(\mu,\nu)} H(\pi | \mu\otimes\nu)<\infty.
\end{align}

We start our analysis with a duality result.

\begin{lemma}\label{lem:duality}
We have
\begin{align*}
&\inf_{\pi\in \widehat{\mathcal{M}}(\mu,\nu)} H(\pi| \mu\otimes\nu)\\
&=\sup_{h\in C_b(\R):\  x h(x)\le 0} \inf_{\pi\in \Pi(\mu,\nu)} \int h(x)(y-x)\,\pi(dx,dy)+H(\pi| \mu\otimes\nu)\\
&=\sup_{h\in C_b(\R):\  x h(x)\le 0}  \sup_{f,g\in C_b(\R)} \int f(x)\,\mu(dx)+\int g(y)\,\nu(dy)\\
&\quad -\int e^{f(x)+g(y)-h(x)(y-x)}\,\mu(dx)\nu(dy)+1.
\end{align*}
In the above, $C_b(\R)$ can be replaced by $L^{1}(\mu)$ and $L^1(\nu)$ respectively.
\end{lemma}

\begin{proof}
We first claim that 
\begin{align}\label{eq:dualityClaim1}
\inf_{\pi\in \widehat{\mathcal{M}}(\mu,\nu)} H(\pi| \mu\otimes\nu)
&=\inf_{\pi\in \Pi(\mu,\nu)}\sup_{h\in C_b(\R):\  x h(x)\le 0} \int h(x)(y-x)\,\pi(dx,dy)+H(\pi| \mu\otimes\nu).
\end{align}
Here ``$\geq$'' holds because $\int h(x)(y-x)\,\pi(dx,dy)\leq0$ for all $\pi\in \widehat{\mathcal{M}}(\mu,\nu)$ 
and $h\in C_b(\R)$ with $x h(x)\le 0$. To see the inequality ``$\leq$'' we fix 
$\pi\in \Pi(\mu,\nu)\setminus\widehat{\mathcal{M}}(\mu,\nu)$ and show that there exists $h\in C_b(\R)$ with $x h(x)\le 0$ such that $\int h(x)(y-x)\,\pi(dx,dy)>0$. By scaling, this will show that the supremum has value $+\infty$. 

Indeed, let $\varphi(x):=\int (y-x)\,\pi_x(dy)$ and $A:=\{x:\, x \varphi(x) < 0\}$. Then  $\varphi\in L^{1}(\mu)$ and $\mu(A)>0$. As $0\notin A$, either $\mu(A\cap(0,\infty))>0$ or $\mu(A\cap(-\infty,0))>0$. We treat the first case; the second is analogous. Consider the finite signed measure $d\rho=\varphi\,d\mu$ on $(0,\infty)$. By our assumption, its negative part $\rho_{-}\neq0$. As $C_b(0,\infty)$ is distribution-determining and $0\neq\rho_{-}\neq\rho_{+}$, there exists a nonnegative $s\in C_b(0,\infty)$ with $\int s \,d\rho_{-} > \int s \,d\rho_{+}$. Moreover, we can choose $s$ such that $s(0+)=0$ (as $\rho$ has no atom at zero), and then~$s$ can be extended to a bounded continuous function~on $\R$ that vanishes on $(-\infty,0]$. Taking $h(x):=-s(x)$, we have $xh(x)\leq0$ and
\begin{align*}
  \int h(x)(y-x)\,\pi(dx,dy) =  \int h(x)\varphi(x)\,\mu(dx)= \int h(x)\, (\rho_{+}-\rho_{-})(dx) >0,
\end{align*} 
completing the proof of~\eqref{eq:dualityClaim1}.

Note that the sets $\{h\in C_b(\R):\  x h(x)\le 0\}$ and $\Pi(\mu,\nu)$ are convex, and $\Pi(\mu,\nu)$ is compact for the weak topology. Furthermore $h \mapsto  \int h(x)(y-x)\,\pi(dx,dy)$ is concave and $\pi\mapsto  \int h(x)(y-x)\,\pi(dx,dy)$ is convex and continuous on $\Pi(\mu,\nu)$; the continuity follows from the fact that $(x,y)\mapsto h(x)(y-x)$ is a continuous function with $|h(x)(y-x)|\le C(1+|x|+|y|)$ and the assumption that the marginals~$\mu$,$\nu$ have finite first moments.  Using Sion's minimax theorem \cite[Cor. 3.3]{Sion.58} and~\eqref{eq:dualityClaim1},
\begin{align*}
\inf_{\pi\in \widehat{\mathcal{M}}(\mu,\nu)} H(\pi| \mu\otimes\nu)
& =\inf_{\pi\in \Pi(\mu,\nu)}\sup_{h\in C_b(\R):\  x h(x)\le 0} \int h(x)(y-x)\,\pi(dx,dy)+H(\pi| \mu\otimes\nu)\\
&=\sup_{h\in C_b(\R):\  x h(x)\le 0} \; \inf_{\pi\in \Pi(\mu,\nu)} \int h(x)(y-x)\,\pi(dx,dy)+H(\pi| \mu\otimes\nu).
\end{align*}
We focus on the inner minimization for fixed $h$, which is an EOT problem with cost function $c(x,y):=h(x)(y-x)$. As $|h(x)(y-x)|\le C(1+|x|+|y|)$, we have $c\in L^{1}(\mu\otimes\nu)$ and~$c$ is equivalent to a nonnegative cost after adding $C(1+|x|+|y|)$. EOT duality (specifically \cite[Theorem~4.7 and Remark~4.4]{Nutz.20}) allows us to conclude that the above infimum equals
\begin{align*}
\sup_{f\in L^{1}(\mu),\, g\in L^{1}(\nu)} \int f(x)\,\mu(dx)+\int g(y)\,\nu(dy) 
 -\int e^{f(x)+g(y)-h(x)(y-x)}\,\mu(dx)\nu(dy)+1.
\end{align*}
Here we can replace $L^{1}$ with $L^{\infty}$ by monotone approximation, and then further replace $L^{\infty}$ with $C_b(\R)$ by a standard density argument, completing the proof.
\end{proof}

\begin{corollary}
Define
\begin{align*}
L_m(\R):= \{h:\R\to [-m,m]: h\mathrm{\ is\ }m\text{-}\mathrm{Lipschitz}, \,  xh(x)\le 0\}.
\end{align*}
Then
\begin{align}\label{eq:minimax}
\begin{split}
&\inf_{\pi\in \widehat{\mathcal{M}}(\mu,\nu)} H(\pi| \mu\otimes\nu)\\
&=\sup_{m\in \N} \; \sup_{h\in L_m(\R)} \; \inf_{\pi\in \Pi(\mu,\nu)}  \int h(x)(y-x)\,\pi(dx,dy)+H(\pi| \mu\otimes\nu)\\
&=\sup_{m\in \N} \;  \inf_{\pi\in \Pi(\mu,\nu)} \; \sup_{h\in L_m(\R)}  \int h(x)(y-x)\,\pi(dx,dy)+H(\pi| \mu\otimes\nu). 
\end{split}
\end{align}
\end{corollary}

\begin{proof}
By Lemma \ref{lem:duality} and an approximation argument,
\begin{align*}
&\inf_{\pi\in \widehat{\mathcal{M}}(\mu,\nu)} H(\pi| \mu\otimes\nu)\\
&= \sup_{h\in C_b(\R):\  x h(x)\le 0} \inf_{\pi\in \Pi(\mu,\nu)} \int h(x)(y-x)\,\pi(dx,dy)+H(\pi| \mu\otimes\nu)\\
&= \sup_{m\in \N}  \sup_{h\in L_m(\R)} \inf_{\pi\in \Pi(\mu,\nu)}  \int h(x)(y-x)\,\pi(dx,dy)+H(\pi| \mu\otimes\nu).
\end{align*}
Specifically, the approximation argument uses tightness of $\Pi(\mu,\nu)$ and the fact that on any given compact interval, we can approximate a given $h\in C_{b}(\R)$ with $x h(x)\le 0 $ uniformly by Lipschitz functions satisfying the same condition. 

The last part of~\eqref{eq:minimax} now follows by another application of Sion's minimax theorem similar to the proof of Lemma~\ref{lem:duality}.
\end{proof}

\begin{lemma}\label{lem:saddle}
Fix $m\in \N$. The problem
\begin{align*}
\sup_{h\in L_m(\R)} \inf_{\pi\in \Pi(\mu,\nu)}   \int h(x)(y-x)\,\pi(dx,dy)+H(\pi| \mu\otimes\nu)
\end{align*}
has a maximizer $h_m\in L_m(\R)$ and the problem 
\begin{align*}
\inf_{\pi\in \Pi(\mu,\nu)}  \sup_{h\in L_m(\R)}  \int h(x)(y-x)\,\pi(dx,dy)+H(\pi| \mu\otimes\nu)
\end{align*}
has a minimizer $\pi_m\in \Pi(\mu,\nu)$. Moreover, $(h_m,\pi_m)$ is a saddle point:
\begin{align}\label{eq:saddle}
\begin{split}
&\sup_{h\in L_m(\R)} \inf_{\pi\in \Pi(\mu,\nu)}   \int h(x)(y-x)\,\pi(dx,dy)+H(\pi| \mu\otimes\nu) \\
&=  \inf_{\pi\in \Pi(\mu,\nu)}   \int h_m(x)(y-x)\,\pi(dx,dy)+H(\pi| \mu\otimes\nu) \\
&= \int h_m(x)(y-x)\,\pi_m(dx,dy)+H(\pi_m| \mu\otimes\nu)\\
&= \sup_{h\in L_m(\R)}  \int h(x)(y-x)\,\pi_m(dx,dy)+H(\pi_m| \mu\otimes\nu)\\
&=  \inf_{\pi\in \Pi(\mu,\nu)}  \sup_{h\in L_m(\R)}  \int h(x)(y-x)\,\pi(dx,dy)+H(\pi| \mu\otimes\nu).
\end{split}
\end{align}
Finally, we have $\pi_m \in \Pi_m(\mu,\nu)$, where 
\begin{align*}
\Pi_m(\mu,\nu):= \Big\{\pi\in \Pi(\mu,\nu): \sup_{h\in L_m(\R)} \int h(x)(y-x)\,\pi(dx,dy)\le H(\hat\pi|\mu\otimes\nu)\Big\}
\end{align*}
and $\hat{\pi}\in \widehat{\mathcal{M}}(\mu,\nu)$ denotes the unique optimizer of \eqref{eq:opt3}. 
\end{lemma}

\begin{proof}
  By the Arzela--Ascoli theorem, $L_m(\R)$ is compact with respect to uniform convergence on compact subsets. The map 
\begin{align*}
L_m(\R) \ni h \mapsto  \int h(x)(y-x)\,\pi(dx,dy)+H(\pi| \mu\otimes\nu)
\end{align*}
is continuous for that convergence, so that the infimum
\begin{align*}
  L_m(\R) \ni h \mapsto  \inf_{\pi\in \Pi(\mu,\nu)}   \int h(x)(y-x)\,\pi(dx,dy)+H(\pi| \mu\otimes\nu)
\end{align*}
is  upper semicontinuous. The existence of a maximizer $h_m\in L_m(\R)$ follows. Similarly, weak compactness of $\Pi(\mu,\nu)$ and integrability of $\mu,\nu$ yield a minimizer $\pi_m\in \Pi(\mu,\nu)$. To verify the saddle point property, note that
\begin{align*}%
\begin{split}
&\sup_{h\in L_m(\R)} \inf_{\pi\in \Pi(\mu,\nu)}   \int h(x)(y-x)\,\pi(dx,dy)+H(\pi| \mu\otimes\nu) \\
&=  \inf_{\pi\in \Pi(\mu,\nu)}   \int h_m(x)(y-x)\,\pi(dx,dy)+H(\pi| \mu\otimes\nu) \\
&\le \int h_m(x)(y-x)\,\pi_m(dx,dy)+H(\pi_m| \mu\otimes\nu)\\
&\le \sup_{h\in L_m(\R)}  \int h(x)(y-x)\,\pi_m(dx,dy)+H(\pi_m| \mu\otimes\nu)\\
&=  \inf_{\pi\in \Pi(\mu,\nu)}  \sup_{h\in L_m(\R)}  \int h(x)(y-x)\,\pi(dx,dy)+H(\pi| \mu\otimes\nu).
\end{split}
\end{align*}
Using once more Sion's minimax theorem, the order of infimum and supremum can be exchanged, showing that equality holds throughout~\eqref{eq:saddle}. %

It remains to show $\pi_m \in \Pi_m(\mu,\nu)$.
As $xh_m(x)\le 0$ and $x \int (y-x)\hat{\pi}_{x}(dy)\ge 0$, we have
\begin{align*}
\int h_m(x)(y-x)\hat{\pi}(dx,dy)=\int h_m(x)\int (y-x)\,\hat{\pi}_x(dy)\,\mu(dx)\le 0.
\end{align*}
Using the definition of $\pi_{m}$, it follows that
\begin{align*}
H(\hat{\pi}|\mu\otimes\nu)&\ge \int h_m(x)(y-x)\,\hat{\pi}(dx,dy)+H(\hat{\pi}| \mu\otimes\nu) \\
&\ge  \int h_m(x)(y-x)\,\pi_m(dx,dy)+H(\pi_m| \mu\otimes\nu) \\
&\stackrel{\eqref{eq:saddle}}{=} \sup_{h\in L_m(\R)}  \int h(x)(y-x)\,\pi_m(dx,dy)+H(\pi_m| \mu\otimes\nu)\\
&\ge \sup_{h\in L_m(\R)}  \int h(x)(y-x)\,\pi_m(dx,dy). \qedhere
\end{align*}
\end{proof}

The following holds for arbitrary couplings in $\Pi_m(\mu,\nu)$, but will be applied to $\pi_{m}$ as defined in Lemma~\ref{lem:saddle}.

\begin{lemma}\label{lem:com}
Let $\pi_m\in \Pi_m(\mu,\nu)$, $m\geq1$. If $(\pi_{m})$ converges weakly, the limit is in  $\widehat{\mathcal{M}}(\mu,\nu)$.
\end{lemma}

\begin{proof}
Let $\pi_{m}\rightarrow\pi$. We need to prove $x \int (y-x)\,\pi_x(dy) \ge 0$ $\mu$-a.s. It suffices to show
\begin{align*}
\int h(x)(y-x)\,\pi(dx,dy)\le 0 \qquad\text{for all }h\in C_b(\R)\text{ satisfying }  x h(x)\le 0.
\end{align*}
By the definition of $\Pi_m(\mu,\nu)$,
\begin{align*}
 \sup_{h\in L_m(\R)} \int h(x)(y-x)\,\pi_m(dx,dy)\le H(\hat{\pi}|\mu\otimes\nu),
\end{align*}
or equivalently
\begin{align*}
\sup_{h\in L_1(\R)} \int h(x)(y-x)\,\pi_m(dx,dy)\le \frac{H(\hat{\pi}|\mu\otimes\nu)}{m}.
\end{align*}
Letting $m\to\infty$, we deduce for any $h\in L_1(\R)$  that
\begin{align*}
\int h(x)(y-x)\,\pi(dx,dy)&\le \liminf_{m\to \infty}\sup_{\tilde h\in L_1(\R)} \int \tilde h(x)(y-x)\,\pi_m(dx,dy)\\
&\le \liminf_{m\to \infty} \frac{H(\hat{\pi}|\mu\otimes\nu)}{m} = 0. 
\end{align*}
By another scaling argument, $\int h(x)(y-x)\,\pi(dx,dy)\le 0$ for all $h\in L_m(\R)$, $m\in \N$. The claim follows by an  approximation of $h\in C_b(\R)$ satisfying $x h(x)\le 0$ by functions in $L_m(\R)$.
\end{proof}

The next lemma will yield weak pre-compactness of $(h_m)_{m\geq1}$, where $h_{m}\in L_{m}(\R)$ was defined in Lemma~\ref{lem:saddle}.

\begin{lemma}\label{lem:bounded}
Recall $P\in \Pi(\mu,\nu)$ as defined in Lemma \ref{lem:P}. There exists $m_0\in \N$ such that
\begin{align}\label{eq:L1bounded}
\int  |h_m(x)|  \Big|\int (y-x)\,P_x(dy)\Big|\,\mu(dx) \le H(P| \mu\otimes\nu) <\infty \qquad \mbox{for all $m\ge m_0$}.
\end{align}
The sequence $(h_m)$ is bounded in $L^1(\tilde{\mu})$, for the finite nonnegative measure $\tilde{\mu}\sim\mu$ defined by
\begin{align*}
\frac{d\tilde{\mu}}{d\mu}(x):=\left|\int (y-x)\,P_x(dy)\right|.
\end{align*}
\end{lemma}

\begin{proof}
By~\eqref{eq:minimax} and the definition of $h_{m}$ in Lemma~\ref{lem:saddle},
\begin{align*}
\inf_{\pi\in \Pi(\mu,\nu)} \int h_m(x)(y-x)\,\pi(dx,dy)+H(\pi| \mu\otimes\nu) \;\uparrow \;H(\hat{\pi} | \mu\otimes\nu),\qquad m\to \infty.
\end{align*}
Therefore, as $H(\hat{\pi} | \mu\otimes\nu)>0$,
there exists $m_0\in \R$ such that 
\begin{align*}
\inf_{\pi\in \Pi(\mu,\nu)} \int h_m(x)(y-x)\,\pi(dx,dy)+H(\pi| \mu\otimes\nu)\ge 0, \qquad m\ge m_0.
\end{align*}
Next, let $P\in \Pi(\mu,\nu)$ be as in Lemma \ref{lem:P}. We conclude that for $m\ge m_0$,
\begin{align}\label{eq:easy}
\begin{split}
 0&\le \inf_{\pi \in \Pi(\mu,\nu)} \int h_m(x)(y-x)\,\pi(dx,dy)+H(\pi| \mu\otimes\nu)\\
&\le \int h_m(x)(y-x)\,P(dx,dy)+H(P| \mu\otimes\nu).
\end{split}
\end{align}
As $x\int (y-x)\,P_x(dy)\ge 0$ and $xh_m(x)\leq0$, we have
$
  h_m(x) \int (y-x)\,P_x(dy)\le 0
$
and thus~\eqref{eq:easy} becomes
\begin{align*}
\begin{split}
0&\le -\int  |h_m(x)|  \big|\int (y-x)\,P_x(dy)\big|\,\mu(dx)+H(P| \mu\otimes\nu)
\end{split}
\end{align*}
or equivalently~\eqref{eq:L1bounded}. 

As $P\in \Pi(\mu,\nu)$ and $\mu,\nu$ have finite first moments, $\tilde{\mu}$ is a finite measure, and $\tilde{\mu}\sim \mu$ by Lemma~\ref{lem:P}. The bound~\eqref{eq:L1bounded} can then be rephrased as $\|h_m\|_{L^{1}(\tilde{\mu})}\leq H(P| \mu\otimes\nu)<\infty$.
\end{proof}

\begin{corollary}\label{cor:bdd}
The sequence $(h_m)_{m\geq1}$ is bounded in $\mu$-probability.
\end{corollary}

\begin{proof}
  This follows from $\mu\sim\tilde\mu$ and boundedness of $(h_m)$ in $L^{1}(\tilde\mu)$, by a general fact of measure theory. We recall the argument for the convenience of the reader.
  
  As $\mu,\tilde\mu$ are finite equivalent measures, $\varphi:=d\mu/d\tilde\mu\in L^{1}(\tilde\mu)$. Let $\varepsilon>0$. Then $\varphi\in L^{1}(\tilde\mu)$ yields $\delta>0$  such that $\int_{A} \varphi\, d\tilde\mu\leq \varepsilon$ whenever $A$ is a Borel set with $\tilde\mu(A)\leq \delta$. In particular, choosing $K>0$ such that
 $\tilde{\mu}(|h_m|>K)\leq  K^{-1}\|h_m\|_{L^{1}(\tilde{\mu})}\leq \delta$ for all~$m\geq1$,
 \[
   \mu(|h_m|>K) = \int_{\{\|h_m\|>K\}} \varphi \,d\tilde\mu \leq \varepsilon, \qquad m\geq1. \qedhere
 \]
\end{proof}

\begin{lemma}\label{lem:schrodongerEqns}
  The coupling $\pi_m\in\Pi(\mu,\nu)$ defined in Lemma~\ref{lem:saddle} has a density of the form
\begin{align}\label{eq:rn}
\frac{d\pi_m}{d(\mu\otimes\nu)}(x,y)=e^{f_m(x)+g_m(y)-h_m(x)(y-x)}
\end{align} 
for some functions $(f_m,g_m)\in L^{1}(\mu)\times L^{1}(\nu)$ satisfying 
\begin{align}\label{eq:EOTdualityValuem}
  \int f_{m}(x)\,\mu(dx) + \int g_{m}(y)\,\nu(dy) = \int h_m(x)(y-x)\,\pi_m(dx,dy)+H(\pi_m| \mu\otimes\nu)
\end{align} 
and
\begin{align}\label{eq:schroedinger1}
\begin{split}
\int e^{f_m(x)+g_m(y)-h_m(x)(y-x)}\,\mu(dx)=1 \qquad\forall y\in \R,\\
\int e^{f_m(x)+g_m(y)-h_m(x)(y-x)}\,\nu(dy)=1 \qquad\forall x\in \R.
\end{split}
\end{align} 
\end{lemma} 

\begin{proof}
  This follows from the definition of $\pi_{m}$ as the solution of the EOT problem with cost $c(x,y):=h_m(x)(y-x)$. Specifically, we have $c\in L^{1}(\mu\otimes\nu)$ and $c(x,y)\geq -C(1+|x|+|y|)$. The claims then follow from \cite[Remark~4.4, Theorem~4.7, Theorem 4.2, Equation~(4.11) and subsequent discussion]{Nutz.20}. 
\end{proof} 

For later use, we emphasize that the identities in~\eqref{eq:schroedinger1} hold everywhere (not just almost-surely) for the chosen versions of $f_{m},g_{m}$.

Next, we establish bounds and equicontinuity properties of $(f_m,g_m,h_{m})$ which are inspired by \cite{NutzWiesel.21} and crucial for the passage to the limit $m\to\infty$.

\begin{lemma}\label{lem:equi}
Fix $\delta\in (0,1)$. There exist compact sets $\mathcal{X}_{\mathrm{cpt},\delta}, \mathcal{Y}_{\mathrm{cpt},\delta}\subseteq\R$ and Borel sets $$\tilde{A}_{m,\delta}\subseteq A_{m,\delta} \subseteq\mathcal{X}_{\mathrm{cpt},\delta}, \quad B_{m,\delta} \subseteq \mathcal{Y}_{\mathrm{cpt},\delta}\qquad\mbox{with}\qquad \mu(\tilde A_{m,\delta})\wedge \nu(B_{m,\delta})  \ge 1-\delta$$ such that $(h_m\1_{A_{m,\delta}})_{m\in \N}$ is uniformly bounded and for all $m\in \N$,
\begin{align}\label{eq:equi1}
|g_m(y)-g_m(\tilde{y})| &\le \sup_{x\in A_{m,\delta}} |h_m(x)(y-\tilde{y})| -\log(1-\delta)\qquad \forall y, \tilde{y}\in B_{m,\delta},
\end{align}
\begin{align}\label{eq:equi2}
\begin{split}
-\left(\log \int_{B_{m,\delta}} e^{g_{m}(y)-h_m(x)(y-x)} \nu(dy)-\log (1-\delta)\right) & \leq f_{m}(x) \\
& \leq- \log \int_{B_{m,\delta}} e^{g_{m}(y)-h_m(x)(y-x)} \nu(dy) 
\end{split}
\end{align}
for all $x\in \tilde{A}_{m,\delta}$.
\end{lemma}

\begin{proof}
Set $\kappa:=\delta^2<\delta$. Choose compacts $\mathcal{X}_{\text {cpt},\delta}$ and $\mathcal{Y}_{\text{cpt},\delta}$ with $\mu\left(\mathcal{X}_{\text {cpt},\delta}\right) \geq 1-\kappa^2 / 4$ and $\nu\left(\mathcal{Y}_{\text{cpt},\delta}\right) \geq 1-\kappa^2 / 2$. By Corollary~\ref{cor:bdd} there exists $K>0$ such that $\mu(|h_m|> K)\le \kappa^2/4$ for all $m\in \N$. Define $A_{m,\delta}:= \mathcal{X}_{\text {cpt},\delta}\cap \{|h_m|\le K\}$. Then $\mu(A_{m,\delta})\ge 1-\kappa^2/2$. As $\pi_{m} \in \Pi(\mu, \nu)$,
\begin{align}\label{eq:compact}
\pi_{m}\left( A_{m,\delta} \times \mathcal{Y}_{\text {cpt},\delta}\right) \geq 1-\kappa^2 \text {. }
\end{align}
Consider the set
$$
B_{m,\delta}=\left\{y \in \mathcal{Y}_{\mathrm{cpt},\delta}: \int_{A_{m,\delta}} e^{f_{m}(x)+g_{m}(y)-h_m(x)(y-x)} \mu(dx) \geq 1-\kappa\right\}.
$$
We claim that its complement $B_{m,\delta}^c$ satisfies
\begin{align}\label{eq:kappa}
p_{m}:=\nu\left(B_{m,\delta}^c\right) \leq \kappa \text {. }
\end{align}
Indeed, \eqref{eq:schroedinger1} yields
\begin{align}\label{eq:help1}
\int_{A_{m,\delta}} e^{f_{m}(x)+g_{m}(y)-h_m(x)(y-x)} \mu(dx) \leq \int e^{f_{m}(x)+g_{m}(y)-h_m(x)(y-x)} \mu(dx)=1
\end{align}
and thus
\begin{align*}
1-\kappa^2 & \!\!\stackrel{\eqref{eq:compact}}{\leq} \pi_{m}\left(A_{m,\delta} \times \mathcal{Y}_{\mathrm{cpt},\delta}\right)=\int_{\mathcal{Y}_{\mathrm{cpt},\delta}} \int_{A_{m,\delta}}  e^{f_{m}(x)+g_{m}(y)-h_m(x)(y-x)} \mu(d x) \nu(d y) \\
& \leq  \int_{B_{m,\delta}^c} \int_{A_{m,\delta}}   e^{f_{m}(x)+g_{m}(y)-h_m(x)(y-x)} \mu(d x)  \nu(d y)\\
& \quad\:+ \int_{B_{m,\delta}}  \int_{A_{m,\delta}} e^{f_{m}(x)+g_{m}(y)-h_m(x)(y-x)}  \mu(d x)\nu(d y) \\
& \leq (1-\kappa) p_{m}+\left(1-p_{m}\right)=1-p_{m} \kappa,
\end{align*}
which implies \eqref{eq:kappa}. Next, we observe from the definition of $B_{m,\delta}$ that for $y \in B_{m,\delta}$,
\begin{align}\label{eq:key}
\begin{split}
-\left(\log \int_{A_{m,\delta}} e^{f_{m}(x)-h_m(x)(y-x)} \mu(d x)-\log (1-\kappa)\right) & \leq g_{m}(y) \\
& \stackrel{\eqref{eq:help1}}{\leq} \log \int_{A_{m,\delta}} e^{ f_{m}(x)-h_m(x)(y-x)} \mu(d x) .
\end{split}
\end{align}
Let $y, \tilde{y} \in B_{m,\delta}$ and assume without loss of generality that $g_{m}\left(y\right) \geq g_m\left(\tilde y\right)$. Then
\begin{align*}
\left|g_{m}\left(y\right)-g_{m} \left(\tilde y\right)\right| 
& \le \log \int_{A_{m,\delta}} e^{f_m(x)-h_m(x)(\tilde y-x)} \mu(d x)-\log (1-\kappa)  \\
& \quad\; - \log \int_{A_{m,\delta}} e^{f_{m}(x)-h_m(x)(y-x)} \mu(d x) \\
&=\log \int_{A_{m,\delta}} e^{h_m(x)(y-x)-h_m(x)(\tilde y-x)+f_{m}(x)-h_m(x)(y-x)} \mu(d x)- \log (1-\kappa) \\
&\quad\;- \log \int_{A_{m,\delta}} e^{f_m(x)-h_m(x)(y-x)} \mu(dx) \\
& \leq \log \left(e^{\sup _{x \in A_{m,\delta}} \left|h_m(x)(y-x)-h_m(x)(\tilde{y}-x)\right|} \int_{A_{m,\delta}} e^{f_m(x)-h_m(x)(y-x)} \mu(d x)\right) \\
&\quad\;- \log (1-\kappa)- \log \int_{A_{m,\delta}} e^{f_m(x)-h_m(x)(y-x)} \mu(d x) \\
&= \sup _{x \in A_{m,\delta}} \left|h_m(x)(y-\tilde y)\right|- \log (1-\kappa) .
\end{align*}
This concludes the proof of \eqref{eq:equi1}.
Turning to \eqref{eq:equi2}, note that by \eqref{eq:compact}, \eqref{eq:kappa} and the definition of $B_{m,\delta}$,
\begin{align}\label{eq:comapct2}
\begin{split}
\pi_{m}\left(A_{m,\delta} \times B_{m,\delta}\right) & \geq \pi_{m}\left(A_{m,\delta} \times \mathcal{Y}_{\mathrm{cpt},\delta}\right)-\pi_{m}\left(A_{m,\delta}  \times B_{m,\delta}^c\right) \\
& \geq 1-\kappa^2- \int_{B_{m,\delta}^c} \int_{A_{m,\delta}}  e^{f_m(x)+g_m(y)-h_m(x)(y-x)} \mu(d x) \nu(d y) \\
& \geq 1-\kappa^2-\kappa(1-\kappa)=1-\kappa=1-\delta^2,
\end{split}
\end{align}
where we used our definition $\kappa:=\delta^2$ which ensures in particular that $\kappa \in(0, \delta)$. Define
$$
\tilde A_{m,\delta}=\left\{x \in A_{m,\delta}: \int_{B_{m,\delta}} e^{f_m(x)+g_m(y)-h_m(x)(y-x)} \nu(d y) \geq 1-\delta\right\} .
$$
Arguing as for \eqref{eq:kappa} and \eqref{eq:key}, now using \eqref{eq:comapct2} instead of \eqref{eq:compact}, we see that $\mu(\tilde A_{m,\delta}^c) \leq \delta$ and that~\eqref{eq:equi2} holds for $x\in \tilde A_{m,\delta}$.
\end{proof}

\begin{lemma}\label{lem:io}
Recall the sets $\tilde{A}_{m,\delta}$ and $B_{m,\delta}$ from Lemma \ref{lem:equi}. Given $A,B\in\cB(\R)$ with $\mu(A),\nu(B)>\delta$, there exist a subsequence $(m_{k})_{k\geq1}\subseteq\N$ and $x_0,y_0\in\R$ such that $$x_0\in \tilde{A}_{m_{k},\delta}\cap A \quad\text{ and }\quad \quad y_0\in B_{m_{k},\delta}\cap B \quad\text{ for all }\quad k\geq1.$$
\end{lemma}

\begin{proof}
We show the claim for $x_0$; the proof for $y_{0}$ is similar. Define 
\begin{align*}
D_m:= \bigcap_{l=m}^\infty (\tilde A_{l,\delta}\cap A)^c.
\end{align*}
Then $(D_m)_{m\geq1}$ are increasing and the claim is equivalent to $\cap_{m}D_{m}^{c}\neq\emptyset$. %
By Lemma~\ref{lem:equi} we have $\mu(\tilde{A}_{m,\delta})\ge 1-\delta$ and thus $\inf_{m}\mu(\tilde{A}_{m,\delta}\cap A)>0$. As a consequence, 
\begin{align*}
0< \limsup_{m\to \infty} \mu(\tilde{A}_{m,\delta}\cap A) \le \limsup_{m\to\infty}\mu(D_m^c),
\end{align*}
showing that $\mu(\cap_{m}D_{m}^{c})>0$. 
\end{proof}

Applying Lemma \ref{lem:io} with $B=\R$ and taking a subsequence if necessary, we may assume that there exists a common point $y_0\in B_{m,\delta}$ for all $m\in \N$. For the remainder of the proof we choose the normalization 
\begin{align}\label{eq:normalizy0}
  g_m(y_0)=0\quad\text{ for all }\quad m\geq1,
\end{align}
which is achieved by replacing $(f_{m},g_{m})$ with $(f_{m}+g_m(y_0),g_{m}-g_m(y_0))$. We remark that a second normalization has been made implicitly: we have $h_{m}(0)=0$ since $h_{m}$ is continuous and $xh_{m}(x)\le 0$.
For ease of reference, we summarize the consequences of the preceding steps.

\begin{lemma}\label{le:unifBoundsAndTightness}
For any $\delta\in(0,1)$, the sequences
\begin{align*}
  (f_m\1_{\tilde A_{m, \delta}})_{m\geq1}, \quad (g_m\1_{B_{m, \delta}})_{m\geq1}, \quad (h_m\1_{A_{m, \delta}})_{m\geq1} \qquad \mbox{are uniformly bounded.}
\end{align*} 
In particular, $(f_m)_{m\geq1}$ and $(h_m)_{m\geq1}$ are $\mu$-tight, and $(g_m)_{m\geq1}$ is $\nu$-tight.
\end{lemma}

\begin{proof}
Recall from Lemma~\ref{lem:equi} that $(h_m\1_{A_{m, \delta}})$ is uniformly bounded and $\mu(A_{m,\delta})>1-\delta$, showing the claims for $(h_{m})$. (In fact, tightness of  $(h_m)$ is tautological with Corollary~\ref{cor:bdd}.) Lemma~\ref{lem:equi} also states that
\begin{align*}
|g_m(y)-g_m(\tilde{y})| &\le \sup_{x\in A_{m,\delta}} |h_m(x)(y-\tilde{y})| -\log(1-\delta)\qquad \forall y, \tilde{y}\in B_{m,\delta}, ~~ m\in \N,
\end{align*}
that $B_{m,\delta}\subseteq\mathcal{Y}_{\text{cpt},\delta}$ and $\nu(B_{m,\delta})\ge 1-\delta$ for all $m\in \N$. Together with~\eqref{eq:normalizy0} and $y_0\in B_{m,\delta}$ this yields
\begin{align}\label{eq:bdd}
|g_m(y)|\le  \sup_{x\in A_{m,\delta}} |h_m(x)(y-y_{0})| -\log(1-\delta) \qquad \forall y \in B_{m,\delta}, ~~ m\in \N.
\end{align}
As the right-hand side is bounded uniformly in $m\in \N$, the claims for  $(g_{m})$ follow.
For $(f_m)$ we use \eqref{eq:equi2}:
\begin{align*}
-\left(\log \int_{B_{m,\delta}} e^{g_{m}(y)-h_m(x)(y-x)} \nu(dy)-\log (1-\delta)\right) & \leq f_{m}(x) \\
& \leq- \log \int_{B_{m,\delta}} e^{g_{m}(y)-h_m(x)(y-x)} \nu(dy)
\end{align*}
for all $x\in \tilde{A}_{m,\delta}$ and all $m\in \N$. By~\eqref{eq:bdd},  $g_m(y)$ is uniformly bounded on $B_{m,\delta}\subseteq\mathcal{Y}_{\text{cpt},\delta}$. Moreover, $y$ is uniformly bounded on that set by compactness, while $h_m(x)$ and $x$ are uniformly bounded on $A_{m, \delta} \subseteq \mathcal{X}_{\text{cpt}, \delta}$. The claims for $(f_{m})$ now follow from $\tilde A_{m, \delta}\subseteq A_{m, \delta}$ and $\mu(\tilde{A}_{m,\delta})\ge 1-\delta$.
\end{proof}

We shall use the following version of Prokhorov's theorem. (Compared, e.g., with \cite[Theorem~3.2.2, p.\,100]{Durrett.10}, this version constructs the limit~$F$ on $(\R,\mu)$ without extending the probability space, which will important for our application in the proof of Lemma~\ref{lem:weak} below.) We denote by $F_*\mu(A):=\mu(F^{-1}(A))$ the pushforward of the measure~$\mu$ by the Borel function $F:\R\to\R^{d}$.

\begin{lemma}\label{lem:prokhorov}
Let $F_{n}:\R\to\R^{d}$ be Borel and let $(F_n)_{n\geq1}$ be $\mu$-tight. After passing to a subsequence, there exists $F:\R\to\R^{d}$ such that $F_{n}\to F$ weakly under~$\mu$; i.e., $(F_n)_*\mu\rightarrow F_*\mu$. 
\end{lemma}

\begin{proof}
Let $S$ be the set of atoms of $\mu$. By the Lebesgue decomposition theorem, we can decompose $\mu=\mu_{p}+\mu_c$, where $\mu_p=\mu|_{S}$ is a purely atomic measure supported on~$S$ and $\mu_{c}=\mu|_{\R\setminus S}$ is atomless. 

The set $S$ is countable and by tightness, $(F_n(x))_{n}$ is bounded for each $x\in S$. Using a diagonal argument, $(F_n(x))_{n}$ converges for all $x\in S$, after taking a subsequence. If $\mu_c=0$, setting $F(x)=\lim_{n} F_{n}(x)$ for $x\in S$ and $F(x)=0$ for $x\in\R\setminus S$ completes the proof.

Suppose $\mu_c\neq0$. The sequence $(F_n)_*\mu_c =(F_n)_*(\mu-\mu_p)$ is again tight. Applying Prokhorov's theorem to $(F_n)_*\mu_c$ and taking another subsequence, we obtain $(F_n)_*\mu_c\rightarrow \tilde{\mu}$ for a measure $\tilde{\mu}$ with $\tilde{\mu}(\R^{d})=\mu_{c}(\R)$. As $(\R,\mu_{c})$ is atomless, there exists a random vector with law $\tilde{\mu}$; i.e., a Borel function $G:\R\to\R^{d}$ with $G_*\mu_c= \tilde{\mu}$. 

We define $F(x):=\lim_{n} F_{n}(x)$ for $x\in S$ and $F(x):=G(x)$ for $x\in\R\setminus S$.
\end{proof}

\begin{lemma}\label{lem:weak}
  There are Borel functions $\hat{f},\hat{g},\hat{h}:\R\to\R$ such that, after taking a subsequence, $(f_m,h_m,\mathrm{id}_\R)\to (\hat{f},\hat{h},\mathrm{id}_\R)$ weakly under $\mu$ and $(g_m,\mathrm{id}_\R)\to (\hat{g},\mathrm{id}_\R)$ weakly under~$\nu$.
    
Moreover, the functions $U_{m}(x,y):=f_m(x)+g_m(y)-h_m(x)(y-x)$ converge weakly under~$(\mu\otimes\nu)$ to $\hat U(x,y):=\hat f(x)+\hat g(y)-\hat h(x)(y-x)$.
\end{lemma}

\begin{proof}
For the first part, applying Lemma~\ref{lem:prokhorov} together with Lemma~\ref{le:unifBoundsAndTightness} for $F_m(x)=(f_m(x), h_m(x), x)$ yields a $\mu$-weak limit $F(x)=:(\hat{f}(x),\hat{h}(x),\hat{a}(x))$. In particular $x_*\mu \to \hat{a}_*\mu$, and thus $\hat{a}_*\mu= x_*\mu$. By the same arguments, $(g_m(y), y)_*\nu \rightarrow (\hat g(y),y)_*\nu.$
Thus by independence, $$(f_m(x),h_m(x),x,g_m(y),y)_*(\mu\otimes\nu)\rightarrow (\hat f(x),\hat h(x),x,\hat g(y),y)_*(\mu\otimes\nu).$$ The claim follows.
\end{proof}

Next, we connect the weak convergence of $f_m,g_{m},h_m$ to weak convergence of~$\pi_m$. The limit $\doublehat{\pi}$ will be identified as $\hat{\pi}$ (and eventually as $\pi^{*}$) in subsequent steps.

\begin{lemma}\label{lem:form}
Define the measure $\doublehat \pi$ on $\R^{2}$ by
\begin{align}\label{eq:doublehatPiDefn}
\frac{d\doublehat \pi}{d(\mu\otimes\nu)}(x,y)=e^{\hat f(x)+\hat g(y)-\hat h(x)(y-x)}.
\end{align}
Then $\pi_m \rightarrow \doublehat{\pi}$ weakly for $m\to\infty$.
\end{lemma}

\begin{proof}
Recall from \eqref{eq:rn} that $\pi_m$ has Radon--Nikodym derivative
\begin{align*}
\frac{d\pi_m}{d(\mu\otimes\nu)}(x,y)=e^{f_m(x)+g_m(y)-h_m(x)(y-x)}.
\end{align*} 
To show $\pi_m \rightarrow \doublehat{\pi}$, we fix an arbitrary Borel set $S\subseteq \R\times \R$ with $(\mu\otimes\nu)(\partial S)=0$ and prove
\begin{align}\label{eq:claim}
\lim_{m\to \infty} |\pi_m(S)-\doublehat \pi(S)|=0
\end{align}
($\partial S$ denotes the boundary of~$S$). 
Indeed, fix $\delta>0$ and recall from Lemma \ref{le:unifBoundsAndTightness} that there exists $C>0$ such that
\begin{align}
g_m&\le C\qquad \text{on } B_{m,\delta} \text{ for all }m\in \N, \label{eq:bound}\\
f_m&\le C\qquad \text{on } \tilde{A}_{m,\delta}\text{ for all }m\in \N, \label{eq:bound2}
\end{align}
and that $\mu(\tilde{A}_{m,\delta})\wedge \nu(B_{m,\delta})\ge 1-\delta$.
Recalling also that $(h_m)$ is uniformly bounded on $A_{m,\delta}$, we have,  after enlarging $C>0$ if necessary,
\begin{align*}
\sup_{m\in \N} \mu(|h_m|\ge \sqrt{C})\le \delta/2, \qquad \mu(|x|\ge \sqrt{C}/2)\vee \nu(|y|\ge \sqrt{C}/2) \le \delta/4,
\end{align*}
and thus
\begin{align}\label{eq:bound3}
\begin{split}
\pi_m( |h_m(x)(y-x)|\ge C)&\le \pi_m( |h_m(x)|\ge \sqrt{C})+ \pi_m( |x-y|\ge \sqrt{C})\\
&\le \mu( |h_m(x)|\ge \sqrt{C}) + \mu(|x|\ge \sqrt{C}/2) + \nu(|y|\ge \sqrt{C}/2)\\
&\le \delta/2+\delta/4+\delta/4=\delta.
\end{split}
\end{align}
Define measures $\pi_m^C$ and $\doublehat \pi^C$ by
\begin{align*}
\frac{d \pi_m^C}{d(\mu\otimes\nu)}(x,y)&:= e^{[f_m(x)+g_m(y) -h_m(x)(y-x)]\wedge 3C},\\
\frac{d \doublehat \pi^C}{d(\mu\otimes\nu)}(x,y) &:= e^{[\hat f(x)+\hat g(y) -\hat h(x)(y-x)]\wedge 3C}.
\end{align*}
We first note that by the monotone convergence theorem,
\begin{align}\label{eq:1}
\int_S e^{[\hat f(x) + \hat g(y) -\hat h(x)(y-x)]\wedge 3C}\, \mu(dx) \nu(dy)\uparrow \int_S e^{\hat f(x)+\hat g(y) -\hat h(x)(y-x)}\, \mu(dx)\nu(dy)
\end{align}
for $C\uparrow \infty$.
On the other hand, using \eqref{eq:bound}--\eqref{eq:bound3},
\begin{align}\label{eq:2}
\begin{split}
\pi_m(S)&\ge \pi_m^C(S)\ge \pi_m^C \big(S\cap \{f_m<C\}\cap \{g_m<C\}\cap \{-h_m(x)(y-x)<C\}\big)\\
&= \pi_m \big(S\cap \{f_m<C\}\cap \{g_m<C\}\cap \{-h_m(x)(y-x)<C\}\big)\\
&\ge \pi_m(S)-\mu(f_m\ge C) -\nu(g_m\ge C)-\pi_m(-h_m(x)(y-x)\ge C)\\
&\ge \pi_m(S)-3\delta.
\end{split}
\end{align}
Lastly, 
\begin{align}\label{eq:3}
\Big|\int_S e^{[f_m(x)+g_m (y)-h_m(x)(y-x)]\wedge 3C} -e^{[\hat f(x)+\hat g(y) -\hat h(x)(y-x)]\wedge 3C}\, \mu(dx)\nu(dy)\Big|\to 0
\end{align}
for $m\to \infty$ by Lemma \ref{lem:weak} and $(\mu\otimes\nu)(\partial S)=0$. Combining \eqref{eq:1}--\eqref{eq:3} yields
\begin{align*}
&\lim_{m\to \infty}|\pi_m(S)-\doublehat \pi(S)|\\
&\le \limsup_{C\to\infty} \limsup_{m\to \infty} \big( |\pi_m(S)-\pi_m^C(S)|+|\pi_m^C(S)-\doublehat \pi^C(S)|+|\doublehat \pi^C(S)-\doublehat \pi(S)|\big) \leq 3\delta.
\end{align*}
As $\delta>0$ was arbitrary,  \eqref{eq:claim} follows and the proof is complete. %
\end{proof}

We now come to the key technical step of identifying the limit and showing integrability.

\begin{lemma}\label{lem:hard} 
We have $(\hat f,\hat g)\in L^1(\mu)\times L^1(\nu)$ and $\hat{h}(x)(y-x)\in L^{1}(\hat{\pi})$. Moreover, 
\begin{gather} 
    \hat\pi =\doublehat{\pi}, \nonumber\\
    H(\hat{\pi}| \mu\otimes\nu)= \int \hat{f}(x)\,\mu(dx)+\int \hat{g}(y)\,\nu(dy), \nonumber\\
     \int \hat{h}(x)(y-x)\,\hat{\pi}(dx,dy)  =0 \quad \text{and}\quad  \mu(\{x: x\hat h(x) \le 0\}) = 1.\nonumber %
\end{gather} 
\end{lemma}

\begin{proof}
\emph{Step 1.} We first show $\doublehat{\pi}=\hat{\pi}$. 
As $\pi_m\rightarrow\doublehat{\pi}$ by Lemma~\ref{lem:form} and $\pi_m\in \Pi_m(\mu,\nu)$ by Lemma~\ref{lem:saddle}, we have $\doublehat{\pi}\in \widehat{\mathcal{M}}(\mu,\nu)$ by Lemma~\ref{lem:com}. Recall from~\eqref{eq:saddle} that
\begin{align}\label{eq:help}
\begin{split}
&\int h_m(x)(y-x)\,\pi_m(dx,dy)+H(\pi_m|\mu\otimes \nu)\\
&= \sup_{h\in L_m(\R)} \int h(x)(y-x)\,\pi_m(dx,dy)+H(\pi_m|\mu\otimes \nu)\\
&\ge H(\pi_m|\mu\otimes \nu).
\end{split}
\end{align}
Using the lower semicontinuity of $\pi\mapsto H(\pi|\mu\otimes\nu)$, we conclude that
\begin{align} \label{eq:inequality}
\begin{split}
H(\doublehat{\pi}|\mu\otimes \nu)&\le \liminf_{m\to \infty} H(\pi_m|\mu\otimes \nu)\\
&\stackrel{\eqref{eq:help}}{\le}  \liminf_{m\to \infty} \int h_m(x)(y-x)\,\pi_m(dx,dy)+H(\pi_m|\mu\otimes \nu) \\
&\stackrel{\eqref{eq:saddle}, \eqref{eq:minimax}}{=} \inf_{\pi\in \widehat{\mathcal{M}}(\mu,\nu)} H(\pi| \mu\otimes\nu)
\end{split}
\end{align}
and hence $\doublehat{\pi}=\hat{\pi}$. 

\emph{Step 2.} 
Our next goal is to show that $(f_m^+, g_m^+)_{m}$ is $(\mu\otimes\nu)$-uniformly integrable. Note that the Schr\"odinger system \eqref{eq:schroedinger1} can be stated as
\begin{align}\label{eq:schroedinger}
\begin{split}
\int e^{f_m(x)-h_m(x)(y-x)}\,\mu(dx)& = e^{-g_m(y)} \qquad \forall y\in \R,\\
\int e^{g_m(y)-h_m(x)y}\,\nu(dy)& = e^{-f_m(x)-h_m(x)x} \qquad \forall x\in \R.
\end{split}
\end{align}
As $\mu$ is centered and $\mu\neq \delta_0$, we have $\mu((-\infty,0))\wedge \mu((0,\infty))>0$. After taking subsequences, Lemmas~\ref{le:unifBoundsAndTightness} and~\ref{lem:io} yield numbers $C<0<D$ such that $(f_m(C),h_m(C),f_m(D), h_m(D))$ is bounded uniformly in $m\in \N$.
Recall from Assumption~\ref{ass:1} that $\nu((-\infty,s-])>0$ and $\nu([s+, \infty))>0$, where $s_-, s_+$ are the left and right endpoints of the support of~$\mu$. Applying again Lemmas~\ref{le:unifBoundsAndTightness} and~\ref{lem:io}, and taking another subsequence, there exist numbers $E\le s_-<s_+\le F$ such that $(g_m(E), g_m(F))$ is bounded uniformly in $m\in\N$.

Recalling that $h_m(x)x\le 0$, we have $h_m(C)\geq0$ and $h_m(D)\leq0$. Using also that 
$\sgn(E-x)\sgn(F-x)\leq 0$ for $x\in [E,F]$, we obtain
\begin{align*}
e^{f_m^+(x)} &\le 1 + e^{f_m(x)-h_m(x)(E-x)} + e^{f_m(x)-h_m(x)(F-x)}\qquad \forall x\in [E,F],\\
e^{g_m^+(y)} &\le 1 + e^{g_m(y)-h_m(C)y} + e^{g_m(y)-h_m(D)y}\qquad \forall y\in \R
\end{align*}
for all $m\in \N$. Combining this with \eqref{eq:schroedinger},
\begin{align*}
\int e^{f_m^+(x)}\,\mu(dx) &\le \int [1 + e^{f_m(x)-h_m(x)(E-x)} + e^{f_m(x)-h_m(x)(F-x)}]\,\mu(dx)\\
&= 1 + e^{-g_m(E)}+e^{-g_m(F)}
\end{align*}
and
\begin{align*}
\int e^{g_m^+(y)}\,\nu(dy) &\le \int [1 + e^{g_m(y)-h_m(C)y} + e^{g_m(y)-h_m(D)y}]\,\nu(dy)\\
&\le 1 + e^{-f_m(C)-h_m(C)C} +e^{-f_m(D)-h_m(D)D}.
\end{align*}
Recalling boundedness of $(f_m(C),h_m(C),f_m(D), h_m(D), g_m(E), g_m(F))$ and the de la Vall\'ee-Poussin theorem, we have shown that $(f_m^+, g_m^+)_{m}$ is $(\mu\otimes\nu)$-uniformly integrable.

\emph{Step 3.} As $(h_m(x),x)_*\mu \rightarrow (\hat{h}(x),x)_*\mu$, we also have $(xh_m(x))_*\mu \rightarrow (x\hat{h}(x))_*\mu$ and thus
\begin{align}\label{eq:clear1}
\mu(\{x: x\hat{h}(x)\le 0\})\ge \limsup_{m\to \infty} \mu( \{x: xh_m(x)\le 0\})=1
\end{align}
by the Portmanteau theorem. Recalling that $\doublehat{\pi}=\hat{\pi}$ is a probability measure, 
\begin{align}\label{eq:clear2}
\int e^{\hat{f}(x)+\hat{g}(y)-\hat{h}(x)(y-x)}\,\nu(dy)\mu(dx)=1.
\end{align}
Combining \eqref{eq:clear1} and \eqref{eq:clear2} with Lemma \ref{lem:duality} yields
\begin{align*}
\int \hat f(x)\,\mu(dx) +\int \hat g(y)\,\nu(dy)&\le 
\sup_{h\in C_b(\R), xh(x)\le 0} \: \sup_{f,g\in C_b(\R)} \int f(x)\,\mu(dx)+\int g(y)\,\nu(dy) \\
&\qquad-\int e^{f(x)+g(y)-h(x)(y-x)}\,\nu(dy)\mu(dx)+1\\
&= H(\hat{\pi}|\mu\otimes\nu).
\end{align*}
On the other hand, by \eqref{eq:inequality} and~\eqref{eq:EOTdualityValuem} we have
\begin{align*}
H(\hat{\pi}|\mu\otimes\nu)
&\le \liminf_{m\to \infty}   \int h_m(x)(y-x)\,\pi_m(dx,dy)+H(\pi_m|\mu\otimes \nu)\\
&= \liminf_{m\to \infty} \int f_m(x)\,\mu(dx) + \int g_m(y)\,\nu(dy) \\
&\le \limsup_{m\to \infty} \int f_m(x)\,\mu(dx) + \limsup_{m\to \infty}\int g_m(y)\,\nu(dy).
\end{align*}
The weak convergence $(f_m)_*\mu\rightarrow \hat{f}_*\mu$ and $(g_m)_*\nu \rightarrow \hat{g}_*\nu$,  together with the $(\mu\otimes\nu)$-uniform integrability of $(f_m^+, g_m^+)$ and the Portmanteau theorem, yields $(\hat f^{+},\hat g^{+})\in L^{1}(\mu)\times L^{1}(\nu)$ and using again the Portmanteau theorem,
\begin{align*}
\limsup_{m\to \infty} \int f_m(x)\,\mu(dx) + \limsup_{m\to \infty}\int g_m(y)\,\nu(dy)
&\le \int \hat f(x)\,\mu(dx) +\int \hat g(y)\,\nu(dy).
\end{align*}
Putting the last three displays together, we see that equality holds throughout. In particular, this yields
\begin{align}\label{eq:hatDualValueFormula}
H(\hat{\pi}|\mu\otimes\nu)&=  \int \hat f(x)\,\mu(dx) +\int \hat g(y)\,\nu(dy)
\end{align} 
and thus $(\hat f,\hat g)\in L^1(\mu)\times L^1(\nu)$.

In general, if $\pi\in\Pi(\mu,\nu)$ satisfies $\pi\ll\mu\otimes\nu$, then $\log(d\pi/d(\mu\otimes\nu))^{-}\in L^{1}(\pi)$ as $x\log x$ is bounded from below. If $H(\pi|\mu\otimes\nu)=\int \log(d\pi/d(\mu\otimes\nu))\,d\pi <\infty$, the positive part is also integrable, so that
$\log(d\pi/d(\mu\otimes\nu))\in L^{1}(\pi)$. Applying this to $\hat\pi=\doublehat\pi$ and using~\eqref{eq:doublehatPiDefn}, we see that $\hat f(x)+ \hat g(y)-\hat{h}(x)(y-x)\in L^{1}(\hat\pi)$, and as $(\hat f,\hat g)\in L^1(\mu)\times L^1(\nu)$, it follows that $\hat{h}(x)(y-x)\in L^{1}(\hat\pi)$. In particular, we can write $H(\hat{\pi}|\mu\otimes\nu)=\int \log(d\hat\pi/d(\mu\otimes\nu))\,d\hat\pi$ as 
\begin{align*}
H(\hat{\pi}|\mu\otimes\nu)&= \int \hat{f}(x)\,\mu(dx)+\int \hat{g}(y)\,\nu(dy)-\int \hat{h}(x)(y-x)\,\hat{\pi}(dx,dy)
\end{align*} 
and conclude by~\eqref{eq:hatDualValueFormula} that $\int \hat{h}(x)(y-x)\,\hat{\pi}(dx,dy) = 0$.
\end{proof}
 
\begin{lemma}\label{lem:martingale}
We have $\hat \pi\in \mathcal{M}(\mu,\nu)$.
\end{lemma}

\begin{proof}
As $\hat\pi\in\widehat{\mathcal{M}}(\mu,\nu)$, we have  $\int (y-x)\,\hat\pi_x(dy) \ge (\leq)\: 0$ for $\mu$-a.e.\ $x\geq (\leq)\:0$. Suppose for contradiction that $\hat\pi\notin\mathcal{M}(\mu,\nu)$; that is, $A=\{x:\,\int (y-x)\,\hat{\pi}_x(dy)\neq 0\}$ satisfies $\mu(A)>0$. 
As $\mu,\nu$ are both centered, $A_{\pm}:=\{x:\int (y-x)\,\hat{\pi}_x(dy)> (<) \;0\}$ then both have positive $\mu$-mass. Moreover, we must have $A_{\pm}\subseteq\R_{\pm}$ $\mu$-a.s. In particular, there exists a non-negligible set of points  $x_2\le 0\le  x_1$ for which 
\begin{align}\label{eq:mean1}
\int (y-x_1)\,\hat\pi_{x_1}(dy)>0 \qquad \text{and} \qquad \int (y-x_2)\,\hat\pi_{x_2}(dy)<0.
\end{align}
As 
\begin{align*}
\int \hat h(x) \int (y-x)\hat \pi_x(dy)\mu(dx)&=\int  \hat h(x)(y-x)\,\hat \pi(dx,dy)=0
\end{align*}
by Lemma \ref{lem:hard},
we obtain that $\hat h(x_1)=\hat h(x_2)=0$ for $\mu$-a.a. $x_1\in A_+$ and $x_2\in A_-$, so that \eqref{eq:mean1} can be stated as 
\begin{align}\label{eq:mean1a}
\int ye^{\hat f(x_1)+\hat g(y)}\,\nu(dy)>x_1 \qquad \text{and} \qquad \int y e^{\hat f(x_2)+\hat g(y)}\,\nu(dy)<x_2.
\end{align}
Combining \eqref{eq:mean1a} and $x_2\le 0\le  x_1$ we obtain
\begin{align}\label{eq:mean3}
\int y e^{\hat f(x_2)+\hat g(y)}\,\nu(dy)<x_2\le x_1<\int ye^{\hat f(x_1)+\hat g(y)}\,\nu(dy)
\end{align}
and in particular $\hat f(x_1)\neq\hat f(x_2)$. 
On the other hand,
\begin{align*}
1=\int e^{\hat f(x_i) +\hat g(y)-\hat h(x_i)(y-x_i)}\nu(dy)=\int e^{\hat f(x_i)+\hat g(y)}\nu(dy)
\end{align*}
for $i=1,2$ implies $\hat f(x_1)=\hat f(x_2)$, a contradiction.
\end{proof}

\begin{proof}[Proof of Theorem \ref{thm:main}: Existence and Characterization]
As $\hat{\pi}\in \mathcal{M}(\mu,\nu)$ by Lemma~\ref{lem:martingale}, it follows from~\eqref{eq:inclusion} that $\pi^\ast=\hat{\pi}$. The density of that coupling has the desired form and properties by Lemmas~\ref{lem:form} and~\ref{lem:hard}, except that it remains to verify $\hat{h}(x)(y-x)\in  L^1(\pi)$ for all $\pi\in \mathcal{M}(\mu,\nu)$ satisfying $H(\pi|\mu\otimes \nu)<\infty$. Indeed, given such $\pi$, \cite[Lemma~1.4]{Nutz.20} yields that $\log(d\pi^{*}/d(\mu\otimes\nu))^{+}\in L^{1}(\pi)$. As $\log(d\pi^{*}/d(\mu\otimes\nu))= \hat{f}(x)+\hat{g}(y)-\hat{h}(x)(y-x)$ with $(\hat f,\hat g)\in L^1(\mu)\times L^1(\nu)$, it follows that $(\hat{h}(x)(y-x))^{-}\in L^{1}(\pi)$. Now the martingale property of $\pi$ and the Fubini--Tonelli theorem for kernels imply that $\int[ \hat{h}(x)(y-x)]^-\,\pi(dx,dy)=\int[ \hat{h}(x)(y-x)]^+\,\pi(dx,dy)$, so that
$\int \hat{h}(x)(y-x)\,\pi(dx,dy)=0$ and in particular $\hat{h}(x)(y-x)\in  L^1(\pi)$. This completes the existence proof.

For the characterization, suppose that $\pi\in \mathcal{M}(\mu,\nu)$ has a density of the form~\eqref{eq:opt1} with $(f,g,h)\in\cD$. Then it follows from Lemma~\ref{le:weakDuality} that $\pi=\pi^{*}$ is the primal optimizer (and also that $(f,g,h)$ are dual optimizers).
\end{proof}

It remains to show the uniqueness part of Theorem \ref{thm:main}. 
Given potentials $(f,g,h)$ and constants $\alpha,\beta\in\R$, the functions 
$f'(x) :=f(x)+ \alpha x + \beta$, $g'(y):=g(y)-\alpha y-\beta$ and $h'(x):= h(x)-\alpha$ satisfy $f'(x)+g'(y)-h'(x)(y-x)=f(x)+g(y)-h(x)(y-x)$, so that $(f',g',h')$ are again potentials. The next result implies that this affine shift is the only source of non-uniqueness, and completes the proof of Theorem~\ref{thm:main}.

\begin{lemma}\label{le:uniqueness}
Let $\nu\neq\delta_{0}$ and $\pi\in\mathcal{P}(\R\times\R)$.  Suppose that
\begin{align*}
  \frac{d\pi}{d(\mu\otimes\nu)}(x,y)=e^{f(x)+g(y)-h(x)(y-x)}=e^{f'(x)+g'(y)-h'(x)(y-x)} \quad\pi\as
\end{align*} 
for some functions $f,f',g,g',h,h':\R\to\R$. Then there are $\alpha,\beta\in\R$ such that 
\begin{align*}
    f'(x) - f(x) = \alpha x + \beta \quad \mu\as, \qquad g(y)-g'(y)=\alpha y+\beta \quad \nu\as,\qquad  h-h'= \alpha \quad \mu\as
\end{align*} 
\end{lemma} 

\begin{proof}
  Note that $\pi\sim \mu\otimes\nu$. Thus we have
  \begin{align}\label{eq:proofUnique1}
    f(x)+g(y)-h(x)(y-x)=f'(x)+g'(y)-h'(x)(y-x), \qquad (x,y)\in A
  \end{align} 
  for some $A\in\cB(\R\times\R)$ with $(\mu\otimes\nu)(A)=1$. In general, given any $A\in\cB(\R\times\R)$ with $(\mu\otimes\nu)(A)=1$, there exist a set $A_{0}\subseteq A$  with $(\mu\otimes\nu)(A_{0})=1$ and a point $(x_{0},y_{0})\in A_{0}$ such that for all $(x,y)\in A_{0}$, we have $(x_{0},y)\in A_{0}$ and $(x,y_{0})\in A_{0}$. The construction can be found in the proof of \cite[Lemma~4.3]{BeiglbockGoldsternMareschSchachermayer.09} or, more explicitly, in \cite[Lemma~4.5]{GhosalNutzBernton.21b}. Fixing $x=x_{0}$ in~\eqref{eq:proofUnique1}, it follows for all $y$ with $(x_{0},y)\in A_{0}$ (and hence $\nu$-a.s.) that
  \begin{align*}
    g(y)-g'(y) = f'(x_{0}) - f(x_{0}) + h(x_{0})(y-x_{0}) -h'(x_{0})(y-x_{0}).
  \end{align*} 
  In particular, $g(y)-g'(y)=\alpha y+\beta$ for some $\alpha,\beta\in\R$. 
  We deduce
    \begin{align*}
    f'(x) - f(x) + [h(x)-h'(x)](y-x) = \alpha y+\beta \qquad (\mu\otimes\nu)\as
  \end{align*} 
  As the support of $\nu$ includes at least two points, it follows that $h-h'= \alpha$ $\mu$-a.s. Then,  we also see that $f' - f - \alpha x =\beta$ $\mu$-a.s.
\end{proof}

Finally, we show the weak duality that was used in the proof of Corollary~\ref{co:duality}.

\begin{lemma}[Weak duality]\label{le:weakDuality}
Let $(\tilde f,\tilde g,\tilde h)\in\cD$ and $\pi\in\cM(\mu,\nu)$. Then
\begin{align*}
H(\pi | \mu\otimes\nu)
\geq \int \tilde f(x)\, \mu(dx) + \int \tilde g(y) \, \nu(dy) - \log \int e^{\tilde f(x)+\tilde g(y)-\tilde h(x)(y-x)}\, \mu(dx)\nu(dy).
\end{align*}
\end{lemma}

\begin{proof}
Let $(\tilde f,\tilde g,\tilde h)\in\cD$ and $\pi\in\cM(\mu,\nu)$. We may assume that $H(\pi|\mu\otimes \nu)<\infty$, so that $\varphi:= \frac{d\pi}{d(\mu\otimes\nu)}$ exists with values in $[0,\infty)$.
We can write
\begin{align*}
\int  e^{\tilde f(x)+\tilde g(y)-\tilde h(x)(y-x)} \, \mu(dx)\nu(dy) &\geq \int  e^{\tilde f(x)+\tilde g(y)-\tilde h(x)(y-x)} \1_{\{\varphi>0\}}\, \mu(dx)\nu(dy)\\
&= \int  e^{\tilde f(x)+\tilde g(y)-\tilde h(x)(y-x) -\log\varphi(x,y)}\,\pi(dx,dy).
\end{align*}
Hence by Jensen's inequality,
\begin{align}\label{eq:jensen_rev}
\begin{split}
& -\log \int  e^{\tilde f(x)+\tilde g(y)-\tilde h(x)(y-x)} \, \mu(dx)\nu(dy)\\
& \leq \int  [-\tilde f(x)-\tilde g(y)+\tilde h(x)(y-x) + \log\varphi(x,y)] \, \pi(dx,dy) \\
& = -\int \tilde f(x)\, \mu(dx) - \int \tilde g(y) \, \nu(dy) +  H(\pi|\mu\otimes \nu),
\end{split}
\end{align}
where the last equality used $\pi\in \mathcal{M}(\mu,\nu)$ and the integrability properties of $(\tilde f,\tilde g,\tilde h)\in\cD$ under $\pi\in\cM_{\text{fin}}(\mu,\nu)$, as well as the definition of $H(\pi|\mu\otimes \nu)$.
\end{proof}

\section{Closing Remarks}\label{se:closing}

We conclude with an example regarding Remark~\ref{rk:irred} and Assumption~\ref{ass:1}. While the assumption is satisfied in the example, it shows that the integrability of $h$ is delicate when the potentials functions $\varphi_{\mu}$ and $\varphi_{\nu}$ are not uniformly separated.

\begin{example}\label{ex:strictSeparationHelps}
  Let $b\geq a:=1/2$ and $\nu=(\delta_{-b}+\delta_{b})/2$. Moreover, let $\mu\preceq_c\nu$ satisfy $\supp(\mu)=[-a,a]$ and $\mu(\{\pm a\})=0$. There exists a unique martingale coupling $\pi\in\mathcal{M}(\mu,\nu)$; namely, $\pi=\mu(dx)\otimes\pi_{x}(dy)$ for
  \begin{align*}
    \pi_{x} = \frac{b-x}{2b}\delta_{-b}+\frac{b+x}{2b}\delta_{b}.
  \end{align*} 
  In particular, Assumption~\ref{ass:0} is satisfied.
  Direct calculation yields that $\frac{d\pi}{d(\mu\otimes\nu)}(x,y)=\exp[f(x)+g(y)-h(x)(y-x)]$ for
    \begin{align*}
      f(x)= \log\left(\frac{b-x}{2b}\right) - \frac{b+x}{2b}\log\left(\frac{b-x}{b+x}\right), \qquad
      g(y)=0, \qquad
      h(x)=\frac{1}{2b}\log\left(\frac{b-x}{b+x}\right).
    \end{align*} 
    These functions are unique up to an affine transformation. The functions $f$ and $g$ are always bounded, whereas the function $h$ is bounded on $\supp(\mu)$ when $b>a$ but unbounded for $b=a$.
    
  Focusing on the boundary case $b=a$, then more specifically, $h$ is unbounded from below and from above, while $c(x,y):=h(x)(y-x)$ is bounded from below but unbounded from above. This leads to $ce^{-c}$ being bounded and in particular $c\in L^{1}(\pi)$. Assumption~\ref{ass:1} is satisfied due to the atoms of~$\nu$, hence $c\in L^{1}(\pi)$ also follows from Theorem~\ref{thm:main}.

Next, we specialize to~$\mu$ defined by
$$\frac{d\mu}{dx}=\frac{C}{(a-|x|) \log(a-|x|)^2}\1_{[-a,a]}(x),$$
where $C=-\log(a)/2$ is the normalizing constant.  Then we have $h\notin L^{1}(\mu)$ as
\begin{align*}
\int |h(x)|\,\mu(dx) &= -2\int_{0}^a \frac{1}{2a} \log(\frac{a-x}{a+x}) \frac{C}{(a-x) \log(a-x)^2}\,dx\\
&\ge -\frac{1}{a} \int_0^a \log( \frac{a-x}{a}) \frac{C}{(a-x) \log(a-x)^2}\,dx\\
&= \frac{1}{a} \log(a) \int_0^a \frac{C}{(a-x) \log(a-x)^2}\,dx -\frac{1}{a} \int_0^a \frac{C}{(a-x) \log(a-x)}\,dx\\
&= \frac{\log(a)}{2a} + \infty,
\end{align*}
and a similar calculation shows that $c(x,y):=h(x) (x-y)\notin L^{1}(\mu\otimes\nu)$.

In the case $b=a$, the potential functions $\varphi_{\mu}$ and $\varphi_{\nu}$ (cf.\ Remark~\ref{rk:irred}) are not uniformly separated. While the example is covered by Assumption~\ref{ass:1} due to the presence of atoms, the boundary case illustrates that the integrability of the potentials can be delicate.
\end{example}

\bibliographystyle{abbrv}
\bibliography{stochfin}

\end{document}